\documentclass[10pt]{article}

\usepackage{amsfonts}
\usepackage{amssymb}
\usepackage{amsthm}
\usepackage{amsmath}
\usepackage{bm}
\usepackage{amscd} 
\usepackage{tikz-cd}

\usepackage{comment}
\usepackage{amsfonts}
\usepackage{amssymb}
\usepackage{amsthm}
\usepackage{amsmath}
\usepackage{bm}
\usepackage{amscd} 
\usepackage{graphicx}
\usepackage{mathrsfs}
\usepackage{amsrefs}
\usepackage{MnSymbol}
\usepackage{hyperref}

\usepackage[all]{xy}

\usepackage{hyperref}

\numberwithin{equation}{section} \DeclareMathSizes{2}{10}{12}{13}
\newtheorem{thm}{Proposition}[section]
\newtheorem{Thm}[thm]{Theorem}
\newtheorem{rem}[thm]{Remark}

\newtheorem{cor}[thm]{Corollary}
\newtheorem{lem}[thm]{Lemma}
\newtheorem{defn}[thm]{Definition}
\newtheorem{prop}[thm]{Proposition}

\newcommand{\ash}{\mbox{\footnotesize{\rotatebox[origin=c]{180}{$!$}}}}

\title{On measurings of algebras over  operads and homology theories
}

\smallskip 

\smallskip 
\author{Abhishek Banerjee \footnote{Dept. of Mathematics, Indian Institute of Science, Bengaluru,  India. Email: abhishekbanerjee1313@gmail.com} \footnote{AB was partially supported by SERB Matrics fellowship MTR/2017/000112} $\qquad$ Surjeet Kour \footnote{Dept. of Mathematics, Indian Institute of Technology, 
Delhi, India. Email: koursurjeet@gmail.com }
\footnote{SK was partially supported by DST-INSPIRE fellowship RES/DST/MATH/SK/201314-044}}
\date{ }

\begin{document}

\maketitle

\medskip

\begin{abstract}
The notion of a coalgebra measuring, introduced by Sweedler, is a kind of generalized ring map between algebras. We begin by studying maps on Hochschild homology induced by coalgebra
measurings. We then introduce a notion of coalgebra measuring between Lie algebras and use it to obtain maps on Lie algebra homology. Further, these measurings between Lie algebras satisfy nice adjoint like properties with respect to universal enveloping algebras.

\smallskip
More generally, we introduce  and undertake a detailed study of the notion of coalgebra measuring between algebras over any operad $\mathcal O$. In case $\mathcal O$
is a binary and quadratic operad, we show that a measuring of $\mathcal O$-algebras leads to maps on operadic homology.  In general, for any operad $\mathcal O$, we construct universal measuring coalgebras to show that the 
category  of $\mathcal O$-algebras is enriched over coalgebras.  We  develop measuring comodules and universal measuring comodules for this theory. We also relate these to measurings of the universal enveloping algebra $U_{\mathcal O}(\mathscr A)$ of an $\mathcal O$-algebra $\mathscr A$ and the modules over it. Finally, we construct the 
Sweedler product $C\rhd \mathscr A$ of a coalgebra $C$ and an $\mathcal O$-algebra $\mathscr A$. The object $C\rhd \mathscr A$ is universal among $\mathcal O$-algebras that arise
as targets of $C$-measurings starting from $\mathscr A$. 
\end{abstract}

\medskip

{\bf \emph{MSC(2010) Subject Classification:}}  16T15, 18D50

\medskip


\medskip

\section{Introduction}

\medskip

Let $K$ be a field and let $A$, $B$ be $K$-algebras. A coalgebra measuring from $A$ to $B$, as defined by Sweedler \cite{Swd}, consists of a $K$-coalgebra $(C,\Delta,\epsilon)$ and a $K$-linear map $\phi: C
\longrightarrow Hom_K(A,B)$ such that
\begin{equation}\label{xeq1.1i}
\phi(c)(ab)=\sum \phi(c_1)(a)\cdot \phi(c_2)(b) \qquad \phi(c)(1_A)=\epsilon(c)\cdot 1_B \qquad a,b\in A,c\in C
\end{equation} where $\Delta(c)=\sum c_1\otimes c_2$. In particular, when $c\in C$ is a grouplike element, i.e., $\Delta(c)=c\otimes c$ and $\epsilon(c)=1$, the map  $\phi(c)\in Hom_K(A,B)$ becomes an ordinary ring 
homomorphism from $A$ to $B$. This leads to the interesting idea that a coalgebra measuring is a kind of generalized ring morphism from $A$ to $B$. In fact, Sweedler \cite{Swd} shows that there is a universal object $\mathcal M(A,B)$ among coalgebra measurings from $A$ to $B$, which gives an enrichment of the category of algebras over the category of coalgebras. Accordingly, the universal measuring coalgebra $\mathcal M(A,B)$ may also be called the ``Sweedler Hom'' (see \cite{AJ}) between algebras. This construction is also closely related to that of the Sweedler dual \cite{Swd} of an algebra. A measuring as in \eqref{xeq1.1i} is said to be cocommutative if the coalgebra $C$ is cocommutative. The cocommutative part $\mathcal M_c(A,B)$
of $\mathcal M(A,B)$ is universal among cocommutative measurings from $A$ to $B$ (see, for instance, \cite{GrM1}).

\smallskip
There is also a version of this idea for modules (see Batchelor \cite{Bat}).  Let $M$ be a left $A$-module,  $N$  a left $B$-module and $(C,\phi)$ a coalgebra measuring from $A$
to $B$. Then, a comodule measuring over $(C,\phi)$ from $M$ to $N$ consists of a left $C$-comodule $P$ and a morphism $\psi: P\longrightarrow Hom_K(M,N)$ such that
\begin{equation}\label{xeq1.2ij}
\psi(p)(am)=\sum \phi(p_0)(a)\psi(p_1)(m)\qquad a\in A, m\in M, p\in P
\end{equation} Accordingly, there is a universal measuring comodule for each coalgebra measuring $(C,\phi)$ and pair of modules $M$, $N$. 

\smallskip
In recent years, the subject of measurings has been studied extensively in a variety of categorical and enriched categorical contexts (see Hyland, L\'{o}pez Franco and Vasilakopoulou \cite{Vas1}, \cite{Vas1.5} and Vasilakopoulou \cite{Vas2}, \cite{Vas3}). Measurings have also been developed as generalized morphisms between bialgebras and Hopf algebras (see 
Grunenfelder and Mastnak \cite{GrM1}, \cite{GrM2}). In \cite{AJ}, Anel and Joyal present the Sweedler Hom as one among several related operations involving measurings of algebras. For generalizations of the Sweedler dual and how they relate to measurings, we refer the reader to Porst and Street \cite{PS}. For more on coalgebra measurings, we refer, for instance, to \cite{Brz}, \cite{LaMa}, \cite{Lebr} and \cite{Take}.

\smallskip
In this paper, we begin with the following simple question: does a coalgebra measuring, seen as a generalized ring map, induce morphisms on Hochschild homology groups? We show that
a cocommutative measuring $(C,\phi)$ of algebras leads to  morphisms of their Hochschild complexes parametrized by elements of $C$. In fact, when the algebras are commutative, this collection of maps induces a coalgebra measuring between Hochschild homology rings with respect to their shuffle product 
structures. Our first main result may be summarized as follows.

\begin{Thm} (see \ref{Prp2.4}, \ref{rTh2.3e}, \ref{Th28y}, \ref{Theorem2.5ft}) Let $A$, $B$ be $K$-algebras and let  $(C_\bullet(A), b)$ and $(C_\bullet(B), b)$ be their respective Hochschild complexes. Let $(C, \phi)$ be a cocommutative measuring from $A$ to $B$.

\smallskip
(a)  For each $c\in C$, we have a  morphism $\tilde \phi(c): (C_\bullet(A),b)\longrightarrow (C_\bullet(B),b)$
of complexes given by:
\begin{equation*} 
\tilde\phi(c)_p:C_p(A)\longrightarrow C_p(B)\qquad \tilde\phi(c)_p(a_0,...,a_p)=c\cdot (a_0,...,a_p):= \sum (\phi(c_1)(a_0),....,\phi(c_{p+1})(a_p))
\end{equation*}

\smallskip
(b) Now let $A$, $B$ be commutative. Then, we have

\begin{itemize}
\item[(1)] The induced $K$-linear map  $ \tilde \phi: C\longrightarrow Hom_K(HH_\bullet(A),HH_\bullet(B))$ is a  measuring of Hochschild homology rings with respect to the shuffle product structures.

\item[(2)] There
is a canonical morphism $\tau(A,B):\mathcal M_c(A,B)\longrightarrow
\mathcal M_c((HH_\bullet(A),sh),(HH_\bullet(B),sh))$ of cocommutative universal measuring 
coalgebras. 

\item[(3)] Let $cALG_K$ (resp. $\widetilde{cALG}_K$) denote the category whose objects are commutative algebras and whose morphism sets are
given by $\mathcal M_c(A,B)$ (resp.  $\mathcal M_c((HH_\bullet(A),sh),(HH_\bullet(B),sh))$). The morphisms $\tau(A,B)$ determine a functor $cALG_K\longrightarrow
\widetilde{cALG}_K$ of categories enriched over coalgebras. 
\end{itemize}

\end{Thm}

\smallskip
We know that the homology of the Chevalley-Eilenberg complex of a Lie algebra behaves in a manner very similar to the Hochschild theory for algebras. As such, we introduce in Section 3 the notion of a coalgebra measuring $(C,\phi)$ between Lie algebras $(\mathfrak g,[.,.])$ and $(\mathfrak g',[.,.]')$:
\begin{equation}\label{xeq1.6r}
\phi: C\longrightarrow Hom_K(\mathfrak g,\mathfrak g') \qquad \phi(c)([a,b])=\sum [\phi(c_1)(a),\phi(c_2)(b)]' \qquad a,b\in \mathfrak g, c\in C
\end{equation} In particular, let $A$, $A'$ be algebras and $(C,\phi)$ a measuring from $A$ to $A'$. We observe that if $C$ is cocommutative, $(C,\phi)$ is also a measuring of Lie algebras, where $A$ and $A'$ are treated as Lie algebras with respect to the standard commutators. 

\smallskip We show that a cocommutative measuring $(C,\phi)$ of Lie algebras from $\mathfrak g$ to $\mathfrak g'$ induces maps on Lie algebra homology
parametrized by elements of $C$. We also describe how measurings of Lie algebras relate to measurings of their universal enveloping algebras.

\begin{Thm} (see \ref{P414.1}, \ref{rP5.2}, \ref{rP5.3}) Let $(\mathfrak g,[.,.])$, $(\mathfrak g',[.,.]')$ be Lie algebras over $K$  and let  $(CE_\bullet(\mathfrak g), d_{CE})$,  $(CE_\bullet(\mathfrak g'), d'_{CE})$ be their respective Chevalley-Eilenberg complexes. Let $(C, \phi)$ be a cocommutative measuring from $\mathfrak g$ to $\mathfrak g'$.

\smallskip
(a) For each $c\in C$, we have a  morphism $\hat \phi(c): (CE_\bullet(\mathfrak g), d_{CE})\longrightarrow(CE_\bullet(\mathfrak g'), d'_{CE})$
of complexes given by:
\begin{equation*} 
\hat\phi(c)_p:CE_p(\mathfrak g)\longrightarrow CE_p(\mathfrak g')\qquad \hat\phi(c)_p(x_1\wedge ...\wedge x_p)=c\cdot (x_1\wedge ...\wedge x_p):= \sum (\phi(c_1)x_1\wedge ...\wedge \phi(c_{p})x_p)
\end{equation*}

\smallskip
(b)  Let   $U(\mathfrak g)$, 
 $U(\mathfrak g')$ denote respectively the universal enveloping algebras of $\mathfrak g$ and $\mathfrak g'$.   Then,
 there is an induced measuring $U(\phi):C\longrightarrow Hom_K(U(\mathfrak g),U(\mathfrak g'))$ of algebras from 
 $U(\mathfrak g)$ to $U(\mathfrak g')$. 
 
 \smallskip
 (c)  Let $A$ be a $K$-algebra. For a fixed cocommutative coalgebra $C$, we have a bijection
\begin{equation*}
LMeas_C(\mathfrak g,A) \cong Meas_C(U(\mathfrak g),A)
\end{equation*} Here,  $LMeas_C(\mathfrak g,A)$ denotes the collection of   measurings of Lie algebras from $\mathfrak g$ to $A$ and $Meas_C(U(\mathfrak g),A)$ denotes the collection
of measurings of algebras from $U(\mathfrak g)$ to $A$. 
\end{Thm}

\smallskip
The expressions in \eqref{xeq1.1i} and \eqref{xeq1.6r} suggest that measurings should be defined more generally for any ``algebraic structure with an operation.''
By considering iterated coproducts 
$\Delta^n:C\longrightarrow C\otimes C \otimes ... \otimes C
$, we observe that the formalism of \eqref{xeq1.1i} and \eqref{xeq1.6r}  may be extended from multiplication or from brackets  to any ``$(n+1)$-ary operation'' (compare \cite[$\S$ 2.6]{Fresse}). We make this precise by introducing
coalgebra measurings between algebras over  a given operad $\mathcal O=\{\mathcal O(n)\}_{n\geq 0}$ in the category of $K$-vector spaces. More explicitly, if $\mathscr A$, $\mathscr B$ are algebras over $\mathcal O$, a coalgebra measuring from $\mathscr A$ to $\mathscr B$ is given by (see Definition \ref{Defmesopalg})
\begin{equation*}
\phi: C\longrightarrow Hom_K(\mathscr A,\mathscr B)
 \qquad \phi(c)\big(\theta(a_1 \otimes \cdots \otimes a_n)\big) = \sum\theta \big(\phi(c_1)(a_1) \otimes \cdots \otimes \phi(c_n)(a_n)\big) 
\end{equation*} for any  $c\in C$, $\theta\in \mathcal O(n)$, $n\geq 0$ and $a_i\in \mathscr A$. From this point onwards, we undertake a detailed study of coalgebra measurings
between algebras over $\mathcal O$. 

\begin{Thm} (see \ref{htm5.4}, \ref{Thm2.4})  Let $\mathcal{O}=\{\mathcal{O}(n)\}_{n\geq 0}$ be an operad and let $\mathscr A$ and $\mathscr B$ be $\mathcal{O}$-algebras. Then, there exists a $K$-coalgebra $\mathcal M=\mathcal M(\mathscr A,\mathscr B)$ and a measuring $\Phi : \mathcal M \longrightarrow Hom_K(\mathscr A, \mathscr B)$ satisfying the following property: given any measuring $(C,\phi)$ from $\mathscr A$ to $\mathscr B$, there exists a unique morphism $\psi: C\longrightarrow \mathcal M$ of coalgebras making the following diagram commutative:
\begin{equation*}
\begin{tikzcd}[column sep=1.5em]
\mathcal{M}  \arrow[rr, , ""{name=U}, , "\Phi"]{} &&  Hom_K(\mathscr A, \mathscr B) \\
 & C \arrow[from=U, phantom, "\scalebox{1.5}{$\circlearrowright$}" description] \arrow{ur}{\phi} \arrow{ul}{\psi}
\end{tikzcd}
\end{equation*}
Further, taking $\mathcal M(\mathscr A,\mathscr B)$ to be the ``Hom-objects,'' the category of $\mathcal{O}$-algebras is enriched over the category of $K$-coalgebras. 

\end{Thm}

If $\mathcal O$ is a binary quadratic operad and $\mathscr A$ is an algebra over $\mathcal O$, we know (see \cite[$\S$ 12.1.2]{LoVa}) that its operadic homology $HH_\bullet^{\mathcal O}(\mathscr A)$
is computed by means of a Koszul complex $C_\bullet^{\mathcal O}(\mathscr A)$.  For the operad  of associative algebras, this recovers the (non-unital) Hochschild homology. For the operad  of Lie algebras, this recovers the Chevalley-Eilenberg homology (see \cite[$\S$ 2.2.3]{Val}). We prove the following result.

\begin{Thm} (see \ref{hTh5.9u}) Let $\mathcal O$ be a binary and quadratic operad.  Let $\mathscr A$, $\mathscr B$ be $\mathcal O$-algebras and let  $(C_\bullet^{\mathcal O}(\mathscr A), d)$ and $(C_\bullet^{\mathcal O}(\mathscr B), d)$ be their respective Koszul complexes. Let $(C, \phi)$ be a cocommutative measuring from $\mathscr A$ to $\mathscr B$.
Then, 
  for each $c\in C$, we have a  morphism $\tilde \phi(c):(C_\bullet^{\mathcal O}(\mathscr A), d)\longrightarrow(C_\bullet^{\mathcal O}(\mathscr B), d)$
of complexes given by:
\begin{equation*} 
\begin{array}{c}
\tilde\phi(c)_n:C_n^{\mathcal O}(\mathscr A)=\mathcal O^{\ash}(n+1)\otimes_{S_{n+1}}\mathscr A^{\otimes n+1}\longrightarrow C_n^{\mathcal O}(\mathscr B)=\mathcal O^{\ash}(n+1)\otimes_{S_{n+1}}\mathscr B^{\otimes n+1}\\  \tilde\phi(c)_n(\delta \otimes (a_1,...,a_{n+1})):= \delta \otimes (\sum  (\phi(c_1)(a_1),....,\phi(c_{n+1})(a_{n+1}))) \\
\end{array}
\end{equation*} In particular, we have an induced morphism $\tilde \phi(c): HH_\bullet^{\mathcal O}(\mathscr A)\longrightarrow HH_\bullet^{\mathcal O}(\mathscr B)$ on operadic
homologies for each $c\in C$. 

\end{Thm}

\smallskip
If $\mathscr A$ is an algebra over an operad $\mathcal O$, we  recall 
(see \cite[$\S$ 1.6]{GK}) that its universal enveloping algebra $U_{\mathcal O}(\mathscr A)$ is an ordinary associative algebra such that
modules over $\mathscr A$ correspond to ordinary left modules over $U_{\mathcal O}(\mathscr A)$.  This algebra is generated by symbols $Z(\theta; a_1,...,a_{m})$, $\theta\in \mathcal O(m+1)$, 
$a_1$,...,$a_{m}\in \mathscr A$, $m\geq 0$ subject to certain relations which we recall in \eqref{gkrel}. We show that a cocommutative measuring of $\mathcal O$-algebras induces
a measuring of their universal enveloping algebras.

\begin{Thm} (see \ref{rP6.9}) Let $\mathscr A$, $\mathscr B$ be $\mathcal O$-algebras and let $(C,\phi)$ be a  cocommutative  measuring from $\mathscr A$ to $\mathscr B$. Then, the  induced
linear map
\begin{equation*}
\begin{array}{c}
U_{\mathcal O}(\phi):C\longrightarrow Hom_K(U_{\mathcal O}(\mathscr A),U_{\mathcal O}(\mathscr B)) \\ U_{\mathcal O}(\phi)(c)(Z(\theta; a_1,...,a_{m})):=Z(\theta,\phi(c_1)(a_1),...,\phi(c_{m})(a_m))=Z(\theta;c_1\cdot a_1,...,c_m\cdot a_m) \\
\end{array}
\end{equation*} is a measuring of algebras from $U_{\mathcal O}(\mathscr A)$ to $U_{\mathcal O}(\mathscr B)$.
\end{Thm}

In Section 7, we introduce comodule measurings between modules over $\mathcal O$-algebras. If $\mathscr M$ is an $\mathscr A$-module and $\mathscr N$ is a $\mathscr B$-module,
a measuring $(P,\psi)$ from $\mathscr M$ to $\mathscr N$ consists of a left $C$-comodule $P$ and a $K$-linear map $ \psi: P \longrightarrow Hom_K(\mathscr M,\mathscr  N)$ satisfying  (see Definition \ref{Defmescomod})
\begin{equation*}
\psi(p)\big(\theta_{\mathscr  M} (a_1 \otimes \cdots \otimes a_{n-1}\otimes m)\big)= \sum \theta_{\mathscr  N }\big( \phi(p_0)(a_1) \otimes \cdots \otimes \phi(p_{n-2})(a_{n-1}) \otimes \psi(p_{n-1})(m)\big)
\end{equation*} for any $p\in P$,  $a_i\in \mathscr A$, $\theta\in \mathcal O(n)$, $n\geq 1$ and $m\in M$. We then prove the following two main results.

\begin{Thm} (see \ref{Th3.2}, \ref{P7.6or}, \ref{T7.7}) Let $\mathscr A$ and $\mathscr B$ be $\mathcal{O}$-algebras. Let $\mathscr M$ be an $\mathscr A$-module, $\mathscr N$ be a $\mathscr B$-module and let $(C, \phi)$ be a measuring from $\mathscr A$ to $\mathscr  B$. 

\smallskip
(a) There
exists  a measuring $(C, \phi)$-comodule $(\mathcal Q_C=\mathcal Q_C(\mathscr M,\mathscr  N), \Psi: \mathcal Q_C \longrightarrow Hom_K(\mathscr M, \mathscr N))$ satisfying the following property: given any measuring  $(C, \phi)$-comodule $(P,\psi)$ from $\mathscr M$ to $\mathscr N$, there exists a unique morphism $\xi: P \longrightarrow \mathcal Q_C$ of $C$-comodules making the following diagram commutative.
\begin{equation*}
\begin{tikzcd}[column sep=1.5em]
\mathcal{Q}_C   \arrow[rr, , ""{name=U}, , "\Psi"]{} &&  Hom_K(\mathscr M, \mathscr N) \\
 & P \arrow[from=U, phantom, "\scalebox{1.5}{$\circlearrowright$}" description] \arrow{ur}{\psi} \arrow{ul}{\xi }
\end{tikzcd}
\end{equation*}

\smallskip
(b) Let $U_{\mathscr A}(\mathscr M)$ be the left module over $U_{\mathcal O}(\mathscr A)$ corresponding to
$\mathscr M$.  Let $U_{\mathscr B}(\mathscr N)$ be the left module over $U_{\mathcal O}(\mathscr B)$ corresponding to
$\mathscr N$.  Let $P$ be a left $C$-comodule and $\psi:P\longrightarrow Hom_K(\mathscr  M,\mathscr  N)$ a morphism.  Then,
$(P,\psi)$ is a measuring comodule over $(C,\phi:C\longrightarrow
Hom_K(\mathscr  A,\mathscr  B))$ if and only if $(P,\psi)$ is  a measuring comodule over $(C,U_{\mathcal O}(\phi):C\longrightarrow Hom_K(U_{\mathcal O}(\mathscr  A),U_{\mathcal O}(\mathscr  B)))$. In particular,  we have an isomorphism of $C$-comodules
\begin{equation*}
\mathcal Q_C(\mathscr M,\mathscr  N)=Q_C(U_{\mathscr A}(\mathscr M),U_{\mathscr B}(\mathscr N))
\end{equation*} 

\end{Thm}

In Section 8, we introduce the Sweedler product of a coalgebra $C$ with an $\mathcal O$-algebra $\mathscr A$. If $C$ is a coalgebra and $A$ is an ordinary associative
algebra, Anel and Joyal have defined in \cite{AJ} the Sweedler product algebra $C\rhd A$ such that
\begin{equation*}
C\rhd -: \mbox{($K$-algebras)}\longrightarrow \mbox{($K$-algebras)}
\end{equation*} is left adjoint to the functor which takes any $K$-algebra to its convolution algebra with respect to $C$. The algebra $C\rhd A$ is universal among algebras which appear
as targets of $C$-measurings starting from $A$. This allows the Sweedler Hom, the Sweedler product and the convolution algebra to be seen as a package of closely related ``Sweedler operations''
on algebras and coalgebras. 

\smallskip If $C$ is a cocommutative coalgebra, the space $[C,\mathscr A]$ of linear maps from $C$ to $\mathscr A$ carries the structure of an algebra over the operad $\mathcal O$ (see \eqref{convolax}).
In order to form the Sweedler product $C\rhd\mathscr A$, we then consider the free $\mathcal O$-algebra $\mathscr F(C\otimes \mathscr A)$ over the vector space $C\otimes \mathscr A$ and take its quotient over
certain relations described in \eqref{cproda}. We conclude by proving the following two results.

\begin{Thm}(see \ref{Th8.3}, \ref{Th8.4})
 Let $C$ be a coalgebra and let $\mathscr A$ be an $\mathcal O$-algebra. 
 
 \smallskip
 (a)  There is an $\mathcal O$-algebra $C\rhd \mathscr A$ and a  measuring $\phi(C,\mathscr A):C\longrightarrow Hom_K(\mathscr A,C\rhd \mathscr A)$ which has the following universal property: given any
measuring $\phi':C\longrightarrow Hom_K(\mathscr A,\mathscr B)$ of $\mathcal O$-algebras, there exists a unique morphism $f:C\rhd \mathscr A\longrightarrow \mathscr B$ 
of $\mathcal O$-algebras such that
$\phi'=Hom(\mathscr A,f)\circ \phi(C,\mathscr A)$.

\smallskip
(b) Suppose that $C$ is cocommutative. 
For $\mathcal O$-algebras $\mathscr A$, $\mathscr B$, there are natural isomorphisms
\begin{equation}
Alg_{\mathcal O}(C\rhd \mathscr A,\mathscr B)=Alg_{\mathcal O}(\mathscr A,[C,\mathscr B])
\end{equation}
\end{Thm}

\smallskip

We mention that after the writing of this paper, we discovered
that some of these ideas on operads have previously been presented
by M. Anel in a talk on joint work with A. Joyal \cite{AJo}.

\smallskip

{\bf Acknowledgements:} We are grateful to Emily Riehl and Anita Naolekar for helpful discussions on operads. The first author is also grateful to the Fields Institute in Toronto, where
part of this paper was written.

\section{Measurings and the shuffle product in Hochschild homology}

Let $K$ be a field and let $A$, $B$ be $K$-algebras. Throughout this paper, we let $C$ be a $K$-coalgebra with coproduct $\Delta: C \longrightarrow C \otimes C$ and counit  $\epsilon : C \longrightarrow K$. We will express the coproduct on $C$ using the standard Sweedler notation $\Delta(c)=\sum c_1\otimes c_2$ for each $c\in C$.

\begin{defn}(see Sweedler \cite{Swd}) 
Let $A$, $B$ be  $K$-algebras. A coalgebra measuring $(C,\phi)$ from $A$ to $B$ consists of a $K$-coalgebra $(C,\Delta,\epsilon)$ and a $K$-linear map $\phi: C
\longrightarrow Hom_K(A,B)$ such that
\begin{equation}\label{eq1.1i}
\phi(c)(aa')=\sum (\phi(c_1)(a))( \phi(c_2)(a'))\qquad \phi(c)(1_A)=\epsilon(c) 1_B.
\end{equation}  for  all $c\in C$ and $a,a'\in A$. For $c\in C$ and $a\in A$, we will often suppress the map $\phi$ and write $\phi(c)(a)$
simply as $c\cdot a$.
\end{defn}

When $c\in C$ is a grouplike element, i.e., $\Delta(c)=c\otimes c$ and $\epsilon(c)=1$, we notice that the condition in \eqref{eq1.1i} implies that $\phi(c)\in Hom_K(A,B)$ is an ordinary ring 
homomorphism from $A$ to $B$. Therefore, a coalgebra measuring may  be seen as a generalized ring morphism from $A$ to $B$. 

\smallskip

In \cite{Swd}, Sweedler  showed that given $K$-algebras $A$ and $B$, there exists a   $K$-coalgebra $\mathcal{M}= \mathcal{M}(A,B)$ and  a measuring $\Phi:  \mathcal{M}  \longrightarrow Hom_K(A,B)$ with the following universal property:
if $(C, \phi)$ is a measuring from $A$ to $B$, then there exists a unique coalgebra morphism $\psi: C \longrightarrow \mathcal{M}$ making the following diagram commute.

\begin{center}
\begin{tikzcd}[column sep=1.5em]
\mathcal{M}   \arrow[rr, , ""{name=U}, , "\Phi"]{} &&Hom_K(A, B) \\
 & C  \arrow[from=U, phantom, "\scalebox{1.5}{$\circlearrowright$}" description] \arrow{ur}{\phi} \arrow{ul}{\psi}
\end{tikzcd}
\end{center}

The coalgebra $\mathcal{M}(A, B)$ is known  as the universal measuring coalgebra and the pair $(\mathcal M(A, B),\Phi)$   is universal among coalgebra measurings    from $A$ to $B$.

\smallskip

If $C$ is a cocommutative coalgebra, then a measuring $(C, \phi)$ from $A$ to $B$ is called a cocommutative measuring. Accordingly, we have a cocommutative coalgebra $\mathcal{M}_c(A, B)$ and a measuring
$\Phi_c: \mathcal{M}_c(A, B)\longrightarrow Hom_K(A,B)$ which is universal among cocommutative measurings from $A$ to $B$. Further, the universal  cocommutative measuring coalgebra $\mathcal{M}_c(A, B)$ is isomorphic to the cocommutative part of the universal measuring coalgebra $\mathcal{M}(A, B)$  (see \cite[Proposition 1.4]{GrM1}). 

\smallskip

We know that a map $A\longrightarrow B$ of $K$-algebras induces a  morphism $HH_\bullet(A)\longrightarrow HH_\bullet(B)$ of  Hochschild homology groups. Since a measuring may be treated as a generalized ring morphism,  we  will now consider morphisms induced on Hochschild homology by coalgebra measurings from $A$ to $B$.

\begin{prop}\label{Prp2.4}
Let $A$ and $B$ be $K$-algebras and let  $(C_\bullet(A), b)$ and $(C_\bullet(B), b)$ be their respective Hochschild complexes. If $(C, \phi)$ is a cocommutative measuring from $A$ to $B$, then for each $c\in C$, we have a  morphism $\tilde \phi(c): (C_\bullet(A),b)\longrightarrow (C_\bullet(B),b)$
of complexes given by:
\begin{equation} 
\tilde\phi(c)_p:C_p(A)\longrightarrow C_p(B)\qquad \tilde\phi(c)_p(a_0,...,a_p)=c\cdot (a_0,...,a_p):= \sum (\phi(c_1)(a_0),....,\phi(c_{p+1})(a_p))
\end{equation}

Here, $c_1, \ldots, c_{p+1}$ are given by the iterated coproduct  $\Delta^p(c)= \sum c_1 \otimes \cdots \otimes c_{p+1}$ and an element
$a_0\otimes\ldots \otimes a_p\in C_p(A)=A^{\otimes (p+1)}$ is denoted by
$(a_0,...,a_p)$.
\end{prop}

\begin{proof}
Since $C$ is cocommutative,  for any $c \in C$ and any  permutation $\sigma \in S_{p+1}$, we must have 
 $\sum c_1\otimes ....\otimes c_{p+1}= \Delta^p(c)=\sum c_{\sigma(1)}\otimes ...\otimes c_{\sigma(p+1)}$. 

\smallskip

We know that  the differential $ b: C_p(A) \longrightarrow C_{p-1}(A)$ is given by $b = \sum_{i=0}^p (-1)^i d_i$ where $d_i(a_0, \ldots, a_p)= (a_0, \cdots, a_ia_{i+1}, \ldots, a_p)$ for $ 0 \leq i \leq p-1$ and $d_p(a_0, \ldots, a_p)=(a_pa_0, \ldots, a_{p-1})$. Therefore, it is enough to show that $ \tilde \phi(c)$ commutes with each $d_i$.

\smallskip
For $0 \leq i \leq p-1$, we have 
\begin{align*}
\tilde \phi(c) \big( d_{i} (a_0,  \ldots, a_p)\big) &=\tilde \phi(c) (a_0,  \ldots,  a_{i}a_{i+1},  \cdots,  a_p)\\
&= \sum\big(\phi(c_1)(a_0), \ldots,  \phi(c_{i+1})(a_{i} a_{i+1}) \ldots, \phi(c_p)(a_p)\big)\\
&=\sum    \big(\phi(c_1)(a_0),  \ldots, \phi(c_{i+1})(a_{i}) \phi(c_{i+2})(a_{i+1}),  \ldots,   \phi(c_{p+1})(a_p)\big)\\
&= d_{i}\Big( \sum  \big(\phi(c_1)(a_0),  \ldots,  \phi(c_{i+1})(a_{i}), \phi(c_{i+2})(a_{i+1}), \ldots,  \phi(c_{p+1})(a_p)\big)\Big)\\
&= d_{i}  \big( \tilde \phi(c)(a_0, \ldots, a_p)\big)
\end{align*}
\smallskip

For $i= p$, we have 
\begin{align*}
\tilde \phi(c)\big( d_p  (a_0,  \ldots,  a_p)\big) &=\tilde \phi(c)(a_pa_0 \otimes \cdots \otimes a_{p-1})\\
&= \sum\big(\phi(c_1)(a_pa_0),  \ldots, \phi(c_p)(a_{p-1})\big)\\
&=\sum \big(\phi(c_1)(a_p)\phi(c_{2})(a_0),  \ldots,   \phi(c_{p+1})(a_{p-1})\big)
\end{align*}

On the other hand, we see that

\begin{align*}
d_{p}  \big( \tilde \phi(c) (a_0, \ldots, a_p)\big)
&=  d_{p}  \Big( \sum \big(\phi(c_1)(a_0),  \phi(c_2)(a_1),   \ldots,   \phi(c_{p+1})(a_p)\big)\Big)\\
&=  \sum  \big( \phi(c_{p+1})(a_p)\phi(c_{1})(a_0),   \ldots,  \phi(c_{p})(a_{p-1})\big)\Big)\\
\end{align*}

Since $C$ is cocommutative, we know that $\sum c_1 \otimes c_2  \otimes \cdots \otimes c_{p+1} = \sum c_{p+1} \otimes c_{1} \otimes c_{2} \otimes \cdots \otimes c_{p}$, which gives
$$\tilde \phi(c)\big( d_p (a_0,  \ldots,  a_p)\big)=d_{p}  \big( \tilde \phi(c)(a_0, \ldots, a_p)\big).$$
\end{proof}

\smallskip

It follows from Proposition \ref{Prp2.4} that each $\tilde \phi(c)$ induces  a map  on homology groups that we continue to denote by $ \tilde\phi(c): HH_{\bullet}(A) \longrightarrow HH_{\bullet}(B)$. In other words, we have a morphism
\begin{equation}\label{hochmes}
\tilde \phi: C\longrightarrow Hom_K(HH_\bullet(A),HH_\bullet(B))
\end{equation}

\smallskip
For $\sigma \in S_p$ and $(a_0, a_1, \ldots, a_p) \in C_p(A)$,  set $\sigma \cdot (a_0, a_1, \ldots, a_p) = (a_0, a_{\sigma^{-1}(1)}, \ldots, a_{\sigma^{-1}(p)})$.
For $K$-algebras $A$ and $B$,  we know (see \cite[$\S$ 4.2]{Loday}) that  there is a $(p,q)$ shuffle product
$$sh_{pq}: C_p(A) \otimes C_q(B) \longrightarrow C_{p+q}(A \otimes B)$$ given by the  formula:

$$(a_0,  \ldots,  a_p) \times (b_0,  \ldots,  b_q)= \sum sgn(\sigma) \sigma \cdot (a_0 \otimes b_0, a_1\otimes 1,  \ldots,  a_p \otimes 1, 1\otimes b_1, \ldots, 1 \otimes b_q)$$

where the sum  runs over all permutations $ \sigma \in S_{p+q}$ such that  $\sigma(1) < \sigma(2)< \cdots < \sigma(p)$ and $\sigma(p+1) < \sigma(p+2)< \cdots < \sigma(p+q)$. 

\smallskip

Now, the shuffle product on complexes is given by:
$$sh: (C_\bullet(A) \otimes C_\bullet(B) )_n = \underset{p+q=n}{\bigoplus} C_p(A) \otimes C_q(B) \longrightarrow C_n(A \otimes B).$$ In particular, if we take $A$ to be a commutative $K$-algebra, then the multiplication $A \otimes A \longrightarrow A$ is a morphism of rings. Composing with the shuffle product, one has
$$sh_{pq} : C_p(A) \otimes C_q(A) \longrightarrow C_{p+q}(A\otimes A)\longrightarrow C_{p+q}(A)$$
given by 
$$(a_0,  \ldots,  a_p) \times (a_0',  \ldots,  a_q')= \sum sign(\sigma) \sigma \cdot (a_0a_0', a_1,  \ldots,  a_p,  a_1', \ldots,  a_q').$$
The shuffle product 
$$sh: (C_\bullet(A) \otimes C_\bullet(A) )_n = \underset{p+q=n}{\bigoplus} C_p(A) \otimes C_q(A) \longrightarrow C_n(A\otimes A)\longrightarrow C_n(A )$$ 
is a morphism of complexes inducing an algebra structure on Hochschild homologies which we denote by $(HH_\bullet(A),sh)$.

\begin{prop}\label{rTh2.3e}
Let $A$ and $B$ be commutative $K$-algebras and let $(C, \phi)$ be a cocommutative measuring from $A$ to $B$. Then, the induced $K$-linear map  $ \tilde \phi: C\longrightarrow Hom_K(HH_\bullet(A),HH_\bullet(B))$ is a cocommutative measuring from $(HH_{\bullet}(A),sh)$ to $(HH_{\bullet}(B),sh)$.
\end{prop}

\begin{proof}
Let $(a_0,  \ldots,  a_p) \in C_p(A)$ and $(a_0',  \ldots,  a_q') \in C_q(A)$. 
Since $C$ is cocommutative, it follows that for $c\in C$ and $\sigma \in S_p$, we have 
$$ c \cdot\big( \sigma \cdot (a_0,  \ldots,  a_p)\big) = \sigma \cdot \big( c \cdot (a_0,  \ldots,  a_p) \big)$$

For $c \in C$, we now see that 

\begin{equation*}
\begin{array}{l}
\tilde \phi(c)  \Big(sh_{pq} \big((a_0,  \ldots,  a_p) \times (a_0',  \ldots,  a_q')\big)\Big) \\ = \tilde \phi(c) \big(\sum {sgn}(\sigma) \sigma \cdot (a_0a_0', a_1, \ldots,  a_p , a'_1,\dots, a_{q}' \big)\\
= \sum {sgn}(\sigma) \sigma \cdot \big(\phi(c_1)(a_0a_0'), \phi(c_2)(a_1), \ldots,  \phi(c_{p+1})(a_p) , \phi(c_{p+2})(a_1'),\ldots, \phi(c_{p+q+1})(a_{q}')\big)\\
=\sum {sgn}(\sigma) \sigma \cdot \big(\phi(c_1)(a_0)\phi(c_2)(a_0'), \phi(c_3)(a_1), \ldots,  \phi(c_{p+2})(a_p) , \phi(c_{p+3})(a_1'),\ldots, \phi(c_{p+q+2})(a_{q}')\big)\\
\end{array}
\end{equation*}

On the other hand, we have

\begin{equation*}
\begin{array}{l}
sh_{pq} \big(\phi(c_1)(a_0,  \ldots,  a_p) \times  \phi(c_2)(a_0',  \ldots,  a_q')\big) \\
= sh_{pq} \Big(  \sum \big(\phi(c_1)(a_0),  \ldots,  \phi(c_{p+1})(a_p) \big) \times \big(\phi(c_{p+2})(a_0'),  \ldots,  \phi(c_{p+q+2})(a_q') \big)\Big)\\
= \sum   {sgn}(\sigma) \sigma \cdot  \big(\phi(c_1)(a_0) \phi(c_{p+2})(a_0'), \phi(c_2)(a_1),...,\phi(c_{p+1})(a_p), \phi(c_{p+3})(a_1') ,..., \phi(c_{p+q+2})(a_{q}')\big)\\
\end{array}
\end{equation*}

Using the fact that $\sum c_1 \otimes c_2  \otimes \cdots \otimes c_{p+q+2} = \sum c_1 \otimes c_{p+2} \otimes c_{2} \otimes \cdots \otimes c_{p+1} \otimes c_{p+3} \otimes \ldots \otimes c_{p+q+2}$, we get
 $$\tilde \phi(c) \cdot \Big(sh_{pq} \big((a_0,  \ldots,  a_p) \times (a_0',  \ldots,  a_q')\big)\Big)=sh_{pq} \big(\phi(c_1)(a_0,  \ldots,  a_p) \times  \phi(c_2)(a_0',  \ldots,  a_q')\big).$$
 Thus, a cocommutative measuring from $A$ to $B$ induces a measuring from $ HH_{\bullet}(A)$ to $HH_{\bullet}(B)$.
\end{proof}

\begin{Thm}\label{Th28y} Let $A$ and $B$ be commutative $K$-algebras. Then, there
is a canonical morphism $\tau(A,B):\mathcal M_c(A,B)\longrightarrow
\mathcal M_c((HH_\bullet(A),sh),(HH_\bullet(B),sh))$ of cocommutative 
coalgebras. 

\end{Thm}

\begin{proof} We consider the universal cocommutative
measuring $\Phi_c:\mathcal M_c(A,B)\longrightarrow Hom_K(A,B)$. Applying Proposition \ref{rTh2.3e} to $\Phi_c$, we obtain a measuring
$\tilde{\Phi}_c: \mathcal M_c(A,B)\longrightarrow Hom_K(HH_\bullet(A),HH_\bullet(B))$ with respect to the shuffle product structures on $HH_\bullet(A)$ and $HH_\bullet(B)$. The result
now follows from the universal property of $\mathcal M_c((HH_\bullet(A),sh),(HH_\bullet(B),sh))$.

\end{proof}

Let $A$, $B$, $D$ be $K$-algebras. Let $\phi:C\longrightarrow Hom_K(A,B)$ be a coalgebra measuring from $A$ to $B$ and let
$\phi':C'\longrightarrow Hom_K(B,D)$ be a coalgebra measuring from $B$ to $D$. Then, it may be easily verified that
the induced morphism
\begin{equation}
\begin{CD}
C\otimes C'@>\phi\otimes \phi'>> Hom_K(A,B)\otimes Hom_K(B,D)@>\circ>> Hom_K(A,D) \\
\end{CD}
\end{equation} is a coalgebra measuring from $A$ to $D$.  By the  property of the universal measuring coalgebra, it follows that the above composition
factors through $\mathcal M(A,D)$. In particular, there is a canonical morphism of coalgebras
\begin{equation}\label{sr2.4qe}
\circ : \mathcal M(A,B)\otimes \mathcal M(B,D)\longrightarrow \mathcal M(A,D)
\end{equation} Additionally, the obvious map $K\longrightarrow Hom_K(A,A)$ given by $t\mapsto t\cdot id_A$ is a measuring, leading
to a coalgebra map $K\longrightarrow \mathcal M(A,A)$. Together with the compositions in \eqref{sr2.4qe}, it follows that $\mathcal M(A,B)$
gives an enrichment of the category of $K$-algebras over the symmetric monoidal category of $K$-coalgebras. Consequently, the coalgebra $\mathcal M(A,B)$
is also known as the ``Sweedler hom'' of two algebras (see, for instance, \cite{AJ}). 

\smallskip Similar to \eqref{sr2.4qe}, there is a canonical morphism of cocommutative coalgebras
\begin{equation*}
\circ : \mathcal M_c(A,B)\otimes \mathcal M_c(B,D)\longrightarrow \mathcal M_c(A,D)
\end{equation*} leading to an enrichment of the category $Alg_K$ of algebras over the category of cocommutative coalgebras. We will denote this enriched category
by $ALG_K$. The full subcategory of $ALG_K$ consisting only of commutative algebras will be denoted
by $cALG_K$.

\smallskip
For commutative algebras $A$, $B$, it is clear that taking the ``hom object''
to be $\mathcal M_c((HH_\bullet(A),sh),(HH_\bullet(B),sh))$  provides another enrichment of the category of commutative algebras  over cocommutative coalgebras.  We will denote
this second enriched category by $\widetilde{cALG}_K$. 

\begin{Thm}\label{Theorem2.5ft} For commutative $K$-algebras $A$, $B$, the morphisms  
\begin{equation*}\tau(A,B):\mathcal M_c(A,B)\longrightarrow
\mathcal M_c((HH_\bullet(A),sh),(HH_\bullet(B),sh))
\end{equation*} induce a functor $cALG_K\longrightarrow \widetilde{cALG}_K$ between categories enriched
over cocommutative coalgebras.
\end{Thm}

\begin{proof}
Let $A$, $B$, $D$ be commutative $K$-algebras. We need to verify that the following diagram  is commutative
\begin{equation}\label{2.6rw}
\begin{CD}
\mathcal M_c(A,B)\otimes \mathcal M_c(B,D)@>\circ >> \mathcal M_c(A,D)\\
@V\tau(A,B)\otimes \tau(B,D)VV @V\tau(A,D)VV \\
\mathcal M_c(HH_\bullet(A),HH_\bullet(B))\otimes \mathcal M_c(HH_\bullet(B),HH_\bullet(D))@>\circ >> \mathcal M_c(HH_\bullet(A),HH_\bullet(D))
\end{CD}
\end{equation} Due to the universal property of $\mathcal M_c(HH_\bullet(A),HH_\bullet(D))$, in order to show that \eqref{2.6rw} commutes, it suffices to 
show that the following two compositions are equal
\begin{equation}\label{2.7rw}
\begin{array}{ccc}
\begin{CD}
\mathcal M_c(A,B)\otimes \mathcal M_c(B,D) \\ @V\circ VV  \\ \mathcal M_c(A,D)  \\ @V\tau(A,D)VV \\  \mathcal M_c(HH_\bullet(A),HH_\bullet(D)) \\ @V\Phi_c(HH_\bullet(A),HH_\bullet(D))VV \\
Hom_K(HH_\bullet(A),HH_\bullet(D))\\
\end{CD} & \qquad &
\begin{CD}
\mathcal M_c(A,B)\otimes \mathcal M_c(B,D) \\
@V\tau(A,B)\otimes \tau(B,D)VV \\
\mathcal M_c(HH_\bullet(A),HH_\bullet(B))\otimes \mathcal M_c(HH_\bullet(B),HH_\bullet(D))\\
@V\circ VV \\
 \mathcal M_c(HH_\bullet(A),HH_\bullet(D))\\
 @V\Phi_c(HH_\bullet(A),HH_\bullet(D))VV \\
Hom_K(HH_\bullet(A),HH_\bullet(D))\\
\end{CD} \\
\end{array}
\end{equation} For the sake of convenience, we denote the left vertical composition in \eqref{2.7rw} by $\phi_1$ and the right vertical composition by $\phi_2$. 

\smallskip
For $x\in \mathcal M_c(A,B)$, $y\in \mathcal M_c(B,D)$ and $(a_0,a_1,...,a_p)\in C_p(A)$, we see that
\begin{equation}\label{2.8rw}
\phi_1(x\otimes y)(a_0,...,a_p)=
(y\circ x)\cdot (a_0,...,a_p) = ((y\circ x)_1\cdot a_0,(y\circ x)_2\cdot a_1,...,(y\circ x)_{p+1}\cdot a_p)
\end{equation} Since $\circ :\mathcal M_c(A,B)\otimes \mathcal M_c(B,D)\longrightarrow \mathcal M_c(A,D)$
is a map of coalgebras, we have $\Delta^p(y\circ x)=\sum (y_1\circ x_1)\otimes ... \otimes (y_{p+1}\circ x_{p+1})$. Combining this with
\eqref{2.8rw}, we have
\begin{equation}\label{2.9rw}
\phi_1(x\otimes y)(a_0,...,a_p)=\sum (y_1\cdot (x_1\cdot a_0),...,y_{p+1}\cdot (x_{p+1}\cdot a_p))
\end{equation} We now note that the following diagram is commutative
\begin{equation}\label{2.10rw}\small 
\begin{CD}
\mathcal M_c(A,B)\otimes \mathcal M_c(B,D) @>\circ (\tau(A,B)\otimes \tau(B,D)) >>  \mathcal M_c(HH_\bullet(A),HH_\bullet(D))\\
@VV\tilde{\Phi}_c(HH_\bullet(A),HH_\bullet(B))\otimes \tilde{\Phi}_c(HH_\bullet(B),HH_\bullet(D))V @V\Phi_c(HH(A),HH(D))VV \\
Hom_K(HH_\bullet(A),HH_\bullet(B))\otimes Hom_K(HH_\bullet(B),HH_\bullet(D)) @>\circ >> Hom_K(HH_\bullet(A),HH_\bullet(D))\\
\end{CD}
\end{equation} Using \eqref{2.10rw}, it follows that the right vertical composition $\phi_2$ in \eqref{2.7rw} may be expressed as
\begin{equation}\label{2.11rw}
\begin{CD}
\mathcal M_c(A,B)\otimes \mathcal M_c(B,D)  \\
@VV\tilde{\Phi}_c(HH_\bullet(A),HH_\bullet(B))\otimes \tilde{\Phi}_c(HH_\bullet(B),HH_\bullet(D))V \\
Hom_K(HH_\bullet(A),HH_\bullet(B))\otimes Hom_K(HH_\bullet(B),HH_\bullet(D)) \\
@VV\circ V \\
 Hom_K(HH_\bullet(A),HH_\bullet(D))\\
\end{CD}
\end{equation} For $x\in \mathcal M_c(A,B)$, $y\in \mathcal M_c(B,D)$ and $(a_0,a_1,...,a_p)\in C_p(A)$, we now see that
\begin{equation}\label{2.12rw}
\begin{array}{ll}
\phi_2(x\otimes y)(a_0,...,a_p)&=
(\tilde{\Phi}_c(HH_\bullet(B),HH_\bullet(D))(y)\circ \tilde{\Phi}_c(HH_\bullet(A),HH_\bullet(B))(x))(a_0,...,a_p)
\\
& = \sum (y_1\cdot (x_1\cdot a_0),...,y_{p+1}\cdot (x_{p+1}\cdot a_p))\\
\end{array}
\end{equation} From \eqref{2.9rw} and \eqref{2.12rw}, it follows that the two compositions $\phi_1$ and $\phi_2$ in \eqref{2.7rw} are equal 
and hence the diagram \eqref{2.6rw} commutes. Finally, since $\Delta^p(1)=\underset{\mbox{\tiny $(p+1)$-times}}{\underbrace{1\otimes 1\otimes ...\otimes 1}}$ in the 
coalgebra $K$, it follows that the composition
\begin{equation}
\begin{CD}
K@>>> \mathcal M_c(A,A)@>\tau(A,A)>> \mathcal M_c(HH(A),HH(A))\\
\end{CD}
\end{equation} coincides with the canonical map $K\longrightarrow  \mathcal M_c(HH(A),HH(A))$. The result follows.
\end{proof}

\begin{cor} For a commutative $K$-algebra $A$, the canonical morphism  
\begin{equation*}\tau(A,A):\mathcal M_c(A,A)\longrightarrow
\mathcal M_c((HH_\bullet(A),sh),(HH_\bullet(A),sh))
\end{equation*} is a morphism of bialgebras.
\end{cor}

\begin{proof} This follows directly from Theorem \ref{Theorem2.5ft} by setting $A=B=D$. 

\end{proof}

\section{Measurings and Lie algebra homology}

Let  $\mathfrak{g}$ be a Lie algebra over $K$. Then,   the classical Chevalley-Eilenberg complex $CE_\bullet(\mathfrak{g})$ of $\mathfrak g$ is defined as follows (see, for instance, \cite[Chapter 4.2.3]{Dale}): 
\begin{equation}\label{CEcomp}
CE_\bullet(\mathfrak{g}):\qquad \wedge^3\mathfrak g\longrightarrow \wedge^2\mathfrak g\longrightarrow\wedge^1\mathfrak g{\longrightarrow} K 
\end{equation}
with differential given by
\begin{equation}\label{cediff}
\begin{array}{c}
d_{CE}: CE_n(\mathfrak g)\longrightarrow CE_{n-1}(\mathfrak g)\\ \\
d_{CE}(x_1\wedge ... \wedge x_n):=\underset{1\leq i< j\leq n}{\sum} (-1)^{i+j+1}[x_i,x_j]\wedge x_1\wedge ... \wedge \hat{x}_i
\wedge ...\wedge \hat{x}_j\wedge ... \wedge x_n\\
\end{array} 
\end{equation} where $\hat{x}_i$ refers to $x_i$ removed from the sequence. For Lie algebras, we know that the homology
$H_\bullet(\mathfrak g)$ of the Chevalley-Eilenberg complex is very similar to the Hochschild theory for algebras (see, for instance, \cite[$\S$   1.1.13, 10.1.5]{Loday}) We will now introduce the notion of a coalgebra measuring between Lie algebras over $K$. 

\begin{defn}
Let $(\mathfrak g,[.,.])$ and $(\mathfrak g',[.,.]')$ be  Lie algebras over $K$. A coalgebra measuring of Lie algebras from $\mathfrak g$ to $\mathfrak g'$  consists of a coalgebra $C$ and a map $\phi: C \longrightarrow Hom_K(\mathfrak g, \mathfrak g')$ such that 
$$ \phi(c)([a, b])= \sum [\phi(c_1)(a), \phi(c_2)(b)]'$$ for all $c\in C$ and $a$, $b\in \mathfrak g$. 
If $C$ is cocommutative, we will say that the above is a cocommutative measuring of Lie algebras. 
\end{defn}

 \smallskip
Unless otherwise mentioned, we will only work with measurings $(C,\phi)$ of Lie algebras where the coalgebra $C$ is cocommutative.   We will now show that a cocommutative
measuring $\phi: C\longrightarrow Hom_K(\mathfrak g, \mathfrak g')$ of Lie algebras leads to a morphism $\hat\phi: C\longrightarrow Hom_K(H_\bullet(\mathfrak g), H_\bullet(\mathfrak g'))$.

\begin{rem}\emph{It is natural to ask whether we can construct universal objects for measurings between Lie algebras over a field $K$, similar to the universal measuring coalgebra for ordinary algebras. In fact, we will carry out this construction in Section 5, working in the full generality of algebras over an operad.}
\end{rem}

\begin{prop}\label{P414.1} Let $(\mathfrak g,[.,.])$, $(\mathfrak g',[.,.]')$ be Lie algebras over $K$  and let  $(CE_\bullet(\mathfrak g), d_{CE})$,  $(CE_\bullet(\mathfrak g'), d'_{CE})$ be their respective Chevalley-Eilenberg complexes. If $(C, \phi)$ is a cocommutative measuring from $\mathfrak g$ to $\mathfrak g'$, then for each $c\in C$, we have a  morphism $\hat \phi(c): (CE_\bullet(\mathfrak g), d_{CE})\longrightarrow(CE_\bullet(\mathfrak g'), d'_{CE})$
of complexes given by:
\begin{equation} \label{eq414.13}
\hat\phi(c)_p:CE_p(\mathfrak g)\longrightarrow CE_p(\mathfrak g')\qquad \hat\phi(c)_p(x_1\wedge ...\wedge x_p)=c\cdot (x_1\wedge ...\wedge x_p):= \sum (\phi(c_1)x_1\wedge ...\wedge \phi(c_{p})x_p),
\end{equation}

where $c_1, \ldots, c_{p}$ are given by the iterated coproduct  $\Delta^{p-1}(c)= \sum c_1 \otimes \cdots \otimes c_{p}$. 

\end{prop}

\begin{proof}
We first claim that $c\cdot (x_1\wedge ...\wedge x_p)$ as defined in \eqref{eq414.13} is well defined. Indeed, for any $\sigma \in S_p$, we know that
\begin{equation}
x_1\wedge ...\wedge x_p={sgn(\sigma)}x_{\sigma(1)}\wedge ...\wedge x_{\sigma(p)}
\end{equation} in $CE_p(\mathfrak g)=\wedge^p\mathfrak g$. Since $C$ is cocommutative, we observe that
\begin{equation}
\begin{array}{ll}
c\cdot (x_1\wedge ...\wedge x_p)&= \sum (\phi(c_1)x_1\wedge ...\wedge \phi(c_{p})x_p)\\
&={sgn(\sigma)}\sum  (\phi(c_{\sigma(1)})x_{\sigma(1)}\wedge ...\wedge \phi(c_{\sigma(p)})x_{\sigma(p)})\\
&={sgn(\sigma)}\sum  (\phi(c_{1})x_{\sigma(1)}\wedge ...\wedge \phi(c_{p})x_{\sigma(p)})\\
&={sgn(\sigma)}c\cdot (x_{\sigma(1)}\wedge ...\wedge x_{\sigma(p)}) \\
\end{array}
\end{equation} We will now show that $d'_{CE}(c\cdot (x_1\wedge ...\wedge x_p))=c\cdot d_{CE}(x_1\wedge ...\wedge x_p)$.  We fix $1\leq i< j\leq p$. Using the cocommutativity of $C$ and the fact that $(C,\phi)$ is a measuring, we now have
\begin{equation*}\small
\begin{array}{l}
c\cdot ([x_i,x_j]\wedge x_1\wedge ... \wedge \hat{x}_i
\wedge ...\wedge \hat{x}_j\wedge ... \wedge x_p) \\ = \sum \phi(c_1)([x_i,x_j])\wedge \phi(c_2)(x_1)\wedge...\wedge \phi(c_i)(x_{i-1})\wedge \phi(c_{i+1})(x_{i+1}) \wedge ... \wedge \phi(c_{j-1})(x_{j-1})\wedge \phi(c_{j})(x_{j+1})\wedge...\wedge \phi(c_{p-1})(x_p)\\
=\sum ([c_1\cdot x_i,c_2\cdot x_j])\wedge c_3\cdot x_1\wedge...\wedge c_{i+1}\cdot x_{i-1}\wedge c_{i+2}\cdot x_{i+1}\wedge ... \wedge c_{j}\cdot x_{j-1}\wedge c_{j+1}\cdot x_{j+1}\wedge...\wedge c_{p}\cdot x_p\\
=\sum ([c_i\cdot x_i,c_j\cdot x_j])\wedge c_1\cdot x_1\wedge...\wedge c_{i-1}\cdot x_{i-1}\wedge c_{i+1}\cdot x_{i+1}\wedge ... \wedge c_{j-1}\cdot x_{j-1}\wedge c_{j+1}\cdot x_{j+1}\wedge...\wedge c_{p}\cdot x_p\\
\end{array}
\end{equation*} The result is now clear from the expression for the differential in \eqref{cediff}.

\end{proof}

From Proposition \ref{P414.1}, it is clear that each $\hat\phi(c)$  induces a morphism on homology groups that we continue to denote 
by $\hat\phi(c)$. This leads to a morphism \begin{equation}\hat\phi : C\longrightarrow Hom_K(H_\bullet(\mathfrak g),H_\bullet(\mathfrak g'))
\end{equation}
Now let $A$ be a $K$-algebra. For each $r\geq 1$, we  let $\mathcal M_r(A)$ denote the ring of $r\times r$-matrices with entries in $A$. There are canonical
inclusions $\mathcal M_r(A)\hookrightarrow \mathcal M_{r+1}(A)$ obtained by adding $0$ entries and we set $\mathcal M(A):=\underset{r\geq 1}{\varinjlim}\textrm{ }\mathcal M_r(A)$. We consider the Lie algebra $\mathfrak{gl}(A)$, which is identical to $\mathcal M(A)$ as a vector space, with
the Lie bracket determined by the standard commutator of matrices in $\mathcal M(A)$. We mention that by Loday-Quillen-Tsygan theorem (see \cite{LQ}, \cite{T}),
the Lie algebra homology of $\mathfrak{gl}(A)$ is isomorphic to the exterior algebra 
\begin{equation*}
H_\bullet(\mathfrak{gl}(A))\cong \bigwedge HC_{\bullet-1}(A)
\end{equation*}
on the shifted cyclic homology $HC_{\bullet-1}(A)$ of $A$.  We conclude this section by showing that  a cocommutative measuring $\phi: C\longrightarrow Hom_K(A,B)$
of algebras induces a morphism $\mathfrak{gl}(\phi): C\longrightarrow Hom_K(H_\bullet(\mathfrak{gl}(A)),H_\bullet(\mathfrak{gl}(B)))$.

\begin{prop}\label{P414.2} Let $A$, $B$ be $K$-algebras and let $\phi:C\longrightarrow Hom_K(A,B)$ be a cocommutative measuring. Then, there is an induced morphism $\mathfrak{gl}(\phi): C\longrightarrow Hom_K(H_\bullet(\mathfrak{gl}(A)),H_\bullet(\mathfrak{gl}(B)))$.

\end{prop}

\begin{proof}
We consider some $r\geq 1$ and a matrix $M=((a_{ij}))\in \mathcal M_r(A)$. It is easily verified that setting $c\cdot M=((c\cdot a_{ij}))$ defines a measuring
of algebras from $\mathcal M_r(A)$ to $\mathcal M_r(B)$. Further, these are compatible with inclusions of matrix rings, thus defining a measuring of algebras 
from $\mathcal M(A)$ to $\mathcal M(B)$.  Since $C$ is also cocommutative, we notice that for matrices $M_1$, $M_2\in \mathcal M(A)$, we have
\begin{equation}
\begin{array}{ll}
c\cdot [M_1,M_2]=c\cdot (M_1M_2-M_2M_1)&=\sum (c_1\cdot M_1)(c_2\cdot M_2)-\sum (c_1\cdot M_2)(c_2\cdot M_1)\\
&=\sum (c_1\cdot M_1)(c_2\cdot M_2)-\sum (c_2\cdot M_2)(c_1\cdot M_1)\\
&=\sum [c_1\cdot M_1,c_2\cdot M_2]\\
\end{array}
\end{equation} In other words, we have an induced measuring of Lie algebras from $\mathfrak{gl}(A)$ to $\mathfrak{gl}(B)$. The result now follows by applying
Proposition \ref{P414.1}. 
\end{proof}

 \section{Lie measurings and the universal enveloping algebra}
 
 As with measurings of algebras, a Lie measuring $\phi:C\longrightarrow Hom_K(\mathfrak g,\mathfrak g')$ may be viewed as a generalized
 morphism of Lie algebras from $\mathfrak g$ to $\mathfrak g'$. In this section, we will describe how Lie measurings are  related to algebra measurings of the universal
 enveloping algebra.  
 
 \begin{lem}\label{rL5.1}  Let $V$, $V'$ be vector spaces and let   $T(V)$, 
 $T(V')$ denote their respective tensor algebras.  Let $C$ be a coalgebra and   $\phi:C\longrightarrow Hom_K(V,V')$ be a linear map. Then, $\phi$  induces a  measuring 
 $T(\phi):C\longrightarrow Hom_K(T(V),T(V'))$ of algebras from $T(V)$ to $T(V')$. 
 
 \end{lem}
 
 \begin{proof}
 By definition, an element of $T(V)$ is a $K$-linear combination of elements of the form
 \begin{equation}
 1\in K \qquad \mbox{and} \qquad x_1\otimes ... \otimes x_p \in V^{\otimes p}, \textrm{ }p\geq 1
 \end{equation}  For $c\in C$, we now define a $K$-linear map $T(\phi)(c):T(V)\longrightarrow T(V')$ by setting
 \begin{equation}
 T(\phi)(c)(1)=\epsilon(c)\cdot 1 \qquad T(\phi)(c)(x_1\otimes ... \otimes x_p)=\sum \phi(c_1)(x_1)\otimes ... \otimes \phi(c_p)(x_p) 
 \end{equation} From the coassociativity of the coalgebra $C$, it is now evident that the above defines a measuring of algebras
 from $T(V)$ to $T(V')$. 
 \end{proof}
 
 \begin{Thm}\label{rP5.2} Let $(\mathfrak g,[.,.])$, $(\mathfrak g',[.,.]')$ be Lie algebras and let   $U(\mathfrak g)$, 
 $U(\mathfrak g')$ denote their respective universal enveloping algebras.  Let $\phi:C\longrightarrow Hom_K(\mathfrak g,\mathfrak g')$ be a cocommutative measuring of Lie algebras. Then,
 there is an induced measuring $U(\phi):C\longrightarrow Hom_K(U(\mathfrak g),U(\mathfrak g'))$ of algebras from 
 $U(\mathfrak g)$ to $U(\mathfrak g')$. 
 \end{Thm}
 
 \begin{proof} Applying Lemma \ref{rL5.1}, we already know that there is a measuring $T(\phi):C\longrightarrow Hom_K(T(\mathfrak g),T(\mathfrak g'))$ of algebras from $T(\mathfrak g)$ to $T(\mathfrak g')$. Composing with the algebra map $T(\mathfrak g')\longrightarrow U(\mathfrak g')$,
 we obtain a measuring from $T(\mathfrak g)$ to $U(\mathfrak g')$ that we continue to denote by $T(\phi):C\longrightarrow Hom_K(T(\mathfrak g),U(\mathfrak g'))$. We claim that any element $z\in T(\mathfrak g)$ of the form
 \begin{equation}\label{tensuniv}
z= x\otimes y - y\otimes x -[x,y]
 \end{equation} with $x$, $y\in \mathfrak g$ satisfies $T(\phi)(c)(z)=0\in U(\mathfrak g')$ for each $c\in C$. Indeed, using the cocommutativity of $C$
 and the fact that  $\phi:C\longrightarrow Hom_K(\mathfrak g,\mathfrak g')$ is a Lie measuring, we have 
 \begin{equation*}
 \begin{array}{ll}
 T(\phi)(c)(x\otimes y - y\otimes x -[x,y])&=\sum (\phi(c_1)(x)\otimes \phi(c_2)(y)-\phi(c_1)(y)\otimes \phi(c_2)(x)) - \phi(c)([x,y])\\
 &=\sum (\phi(c_1)(x)\otimes \phi(c_2)(y)-\phi(c_2)(y)\otimes \phi(c_1)(x)) - \sum [\phi(c_1)(x),\phi(c_2)(y)]'\\
 &=\sum (\phi(c_1)(x)\otimes \phi(c_2)(y)-\phi(c_2)(y)\otimes \phi(c_1)(x) -  [\phi(c_1)(x),\phi(c_2)(y)]')=0 \\
 \end{array}
 \end{equation*} From this, it now follows that if $z'zz''$ is an element of $T(\mathfrak g)$ with $z'$, $z''\in T(\mathfrak g)$ and 
 $z\in T(\mathfrak g)$ an element of the form \eqref{tensuniv}, we must have
 \begin{equation}\label{tens5.4}
 T(\phi)(c)(z'zz'')=\sum( T(\phi)(c_1)(z'))(T(\phi)(c_2)(z))(T(\phi)(c_3)(z''))=0\in U(\mathfrak g')
 \end{equation} for each $c\in C$. In other words, $T(\phi)(c)$ vanishes on the two sided ideal in $T(\mathfrak g)$ generated by elements
 of the form \eqref{tensuniv}. It is now clear that $T(\phi)$ induces a measuring of algebras from $U(\mathfrak g)$ to 
 $U(\mathfrak g')$.
 
 \end{proof}
 
  We now fix a cocommutative coalgebra $C$. For Lie algebras $\mathfrak g$, $\mathfrak g'$, we denote by $LMeas_C(\mathfrak g,\mathfrak g')$ the set of Lie algebra measurings
$C\longrightarrow Hom_K(\mathfrak g,\mathfrak g')$. For $K$-algebras $A$, $B$, we denote by $Meas_C(A,B)$ the set of algebra 
measurings $C\longrightarrow Hom_K(A,B)$.

\begin{Thm}\label{rP5.3} Let $\mathfrak g$ be a Lie algebra and let $A$ be a $K$-algebra. For a fixed cocommutative coalgebra $C$, we have a bijection
\begin{equation}
LMeas_C(\mathfrak g,A) \cong Meas_C(U(\mathfrak g),A)
\end{equation} Here, $A$ is treated as a Lie algebra by taking its standard commutator for multiplication and $U(\mathfrak g)$ is the universal enveloping algebra
of $\mathfrak g$. 

\end{Thm}

\begin{proof} We begin with a Lie measuring $\phi:C\longrightarrow Hom_K(\mathfrak g,A)$. We use this to define
$\psi:C\longrightarrow Hom_K(U(\mathfrak g),A)$ by setting
\begin{equation}\label{tens5.6}
\psi(c)(1)=\epsilon(c)1 \qquad \psi(c)(x_1\otimes ...\otimes x_p)=\sum \phi(c_1)(x_1)...\phi(c_p)(x_p)
\end{equation} for each $c\in C$ and $x_1\otimes ...\otimes x_p\in \mathfrak g^{\otimes p}$, $p\geq 1$. The product on the right hand
side of \eqref{tens5.6} is taken in the algebra $A$. In order to show that $\psi$ is well defined, we need to show that $\psi(c)$ is zero on  the two sided ideal in $T(\mathfrak g)$ generated by elements of the form 
 \begin{equation}\label{tensuniv1}
z= x\otimes y - y\otimes x -[x,y]
 \end{equation} with $x$, $y\in \mathfrak g$. Since  $\phi:C\longrightarrow Hom_K(\mathfrak g,A)$ is a Lie measuring and $C$ is cocommutative, we have in particular that
 \begin{equation}
 \begin{array}{ll}
 \psi(c)( x\otimes y - y\otimes x -[x,y])&=\sum (\phi(c_1)(x)\phi(c_2)(y)-\phi(c_1)(y)\phi(c_2)(x))-\phi(c)([x,y])\\
 &=\sum (\phi(c_1)(x)\phi(c_2)(y)-\phi(c_2)(y)\phi(c_1)(x)-[\phi(c_1)(x),\phi(c_2)(y)])=0\\
 \end{array}
 \end{equation} From reasoning similar to \eqref{tens5.4}, it now follows that $\psi$ is well defined. It is also clear from the definition in 
 \eqref{tens5.6} that $\psi$ is a measuring of algebras.

\smallskip 
On the other hand, suppose that we have a measuring of algebras $\psi':C\longrightarrow Hom_K(U(\mathfrak g),A)$. Then, we define 
\begin{equation}
\phi':C\longrightarrow Hom_K(\mathfrak g,A)\qquad \phi'(c)(x):=\psi'(c)(x)
\end{equation} for $x\in \mathfrak g$ and $c\in C$. It follows that for any $x_1$, $x_2\in \mathfrak g$ and $c\in C$, we have
\begin{equation}
\begin{array}{ll}
\phi'(c)([x_1,x_2])=\psi'(c)([x_1,x_2])&=\psi'(c)(x_1\otimes x_2 - x_2\otimes x_1)\\
&=\sum (\psi'(c_1)(x_1) \psi'(c_2)(x_2)-\psi'(c_1)(x_2) \psi'(c_2)(x_1)) \\
&=\sum (\psi'(c_1)(x_1) \psi'(c_2)(x_2)-\psi'(c_2)(x_2) \psi'(c_1)(x_1)) \\
&=\sum [\psi'(c_1)(x_1),\psi'(c_2)(x_2)]=\sum [\phi'(c_1)(x_1),\phi'(c_2)(x_2)]\\
\end{array}
\end{equation}  It follows that $\phi':C\longrightarrow Hom_K(\mathfrak g,A)$ is a Lie measuring. It is clear that the two associations defined above
are inverse to each other. This proves the result. 
\end{proof}
 
\section{Measurings of algebras over operads}

In this section, we will introduce  coalgebra measurings for algebras over operads. We begin by recalling the definition of a (symmetric) operad in  $Vect_K$, the category
of vector spaces over $K$.

\begin{defn} (see, for instance, \cite[$\S$ 1.1]{Fress}, \cite[Definition 1]{Markl})  An operad in $Vect_K$ is a collection of $K$-vector spaces $\{\mathcal{O}(n)\}_{n\geq 0}$, where $\mathcal O(n)$ is a left
$K[S_n]$-module, together with $K$-linear maps 
$$ \gamma: \mathcal{O}(n) \otimes  \mathcal{O}(k_1) \otimes \cdots  \otimes \mathcal{O}(k_n) \longrightarrow  \mathcal{O}(k_1 + \cdots + k_n) $$

for $n\geq 1$  and a unit map $\eta: K \longrightarrow \mathcal{O}(1)$ such that the following conditions hold true:

\smallskip

1) Associativity:  For  $n \geq 1$ and non-negative integers $m_1, \ldots, m_n$ and $k_1, \ldots, k_m$, where  $m = m_1 +\cdots + m_n$, the following diagram commutes:

\[
\begin{tikzcd}[row sep=large, column sep=10ex]
\big(\mathcal{O}(n) \otimes  \bigotimes_{s=1}^{n} \mathcal{O}(m_s) \big) \bigotimes_{r=1}^{m} \mathcal{O}(k_r) \arrow{r}{\gamma \otimes id} \arrow[swap]{d}{\textit{Shuffle}} & \mathcal{O}(m) \otimes  \bigotimes_{r=1}^{m} \mathcal{O}(k_r)  \arrow{r}{\gamma} & \mathcal{O}(k_1+ \cdots + k_m)  \\
\mathcal{O}(n) \otimes  \bigotimes_{s=1}^{n}\big( \mathcal{O}(m_s) \otimes \bigotimes_{t=1}^{m_s}\mathcal{O}(k_{p_s+t})\big) \arrow{r}{id\otimes  (\bigotimes_{s=1}^{n}\gamma)} & \mathcal{O}(n) \otimes  \left(\bigotimes_{s=1}^{n} \mathcal{O}(q_s) \right) \arrow{ur}{\gamma} 
\end{tikzcd}
\]

where,  for $s \geq 1$,  we have   $p_s = m_1 + \cdots + m_{s-1}$ and $q_s = k_{p_s+1}+ \cdots + k_{p_s+m_s}$.

\smallskip

2) Equivariance:  For permutations $\sigma\in S_n$, $\tau_i\in S_{m_i}$, the following diagrams commute
\begin{equation*}
\begin{CD}
\mathcal O(n)\otimes \left(\mathcal O(m_1)\otimes \ldots \otimes \mathcal O(m_n)\right) @>id\otimes \left(\tau_1\otimes \ldots \otimes\tau_n\right)>> \mathcal O(n)\otimes \left(\mathcal O(m_1)\otimes \ldots \otimes \mathcal O(m_n)\right)\\
@V\gamma VV @V\gamma VV \\
\mathcal O(m_1+\ldots m_n) @>\tau_1\oplus \ldots \oplus \tau_n>> \mathcal O(m_1+\ldots m_n) \\
\end{CD}
\end{equation*} and

\begin{equation*}\small
\begin{tikzcd}[row sep=large, column sep=10ex]
\mathcal O(n)\otimes \left(\mathcal O(m_1)\otimes \ldots \otimes \mathcal O(m_n)\right)\arrow{r}{\sigma \otimes id} \arrow[swap]{d}{id\otimes \sigma^*} &  \mathcal O(n)\otimes \left(\mathcal O(m_1)\otimes \ldots \otimes \mathcal O(m_n)\right)\arrow{r}{\gamma} & \mathcal{O}(m_{1}+\ldots + m_{n})  \\
\mathcal O(n)\otimes \left(\mathcal O(m_{\sigma(1)})\otimes \ldots \otimes \mathcal O(m_{\sigma(n)})\right) \arrow{r}{\gamma} &  \mathcal O(m_{\sigma(1)}+\ldots +m_{\sigma(n)})\arrow{ur}{\sigma_*(m_1,...,m_n)} 
\end{tikzcd}
\end{equation*} Here, $\sigma_*(m_1,...,m_n)$ acts by permuting blocks of length $m_1$, ..., $m_n$. 

\smallskip

3) Unitality:  The following diagrams commute

\begin{equation*}
\begin{array}{ccc}
\begin{tikzcd}
\mathcal{O}(n) \otimes K^{\otimes n} \arrow{r}{\approx} \arrow[swap]{d}{id \otimes \eta^{\otimes n}} &\mathcal{O}(n) \\
\mathcal{O}(n) \otimes \mathcal{O}(1)^{\otimes n} \arrow{ur}{\gamma} 
\end{tikzcd}
& \qquad & 
\begin{tikzcd}
K \otimes \mathcal{O}(n)  \arrow{r}{\approx} \arrow[swap]{d}{\eta \otimes id} &\mathcal{O}(n) \\
\mathcal{O}(1)  \otimes \mathcal{O}(n) \arrow{ur}{\gamma} 
\end{tikzcd}
\end{array}
\end{equation*}

\end{defn}

\begin{defn} (see, for instance, \cite[$\S$ 1.1.13]{Fress}, \cite[$\S$ 3]{Markl})
Let $\mathcal{O}=\{\mathcal{O}(n)\}_{n\geq 0}$ be an operad in $Vect_K$. An  $\mathcal{O}$-algebra is a $K$-vector space $\mathscr A$ together with  maps:
$$ \alpha_{\mathscr A}(n): \mathcal{O}(n) \otimes \mathscr A^{\otimes n} \longrightarrow \mathscr A, ~~ \mbox{for all}~~ n \geq 0$$ such that the following diagrams commute:

\smallskip

1) Associativity:  For  $n,k_1, \ldots, k_n\geq 0$,  we have 

\[
\begin{tikzcd}
\big(\mathcal{O}(n) \otimes \bigotimes_{s=1}^{n} \mathcal{O}(k_s) \big) \otimes  \bigotimes_{s=1}^{n} \mathscr A^{\otimes k_s} \arrow{r}{\gamma \otimes id} \arrow[swap]{d}{\textit{Shuffle}} & \mathcal{O}(k_1+\cdots+k_n) \otimes  \mathscr A^{\otimes ( k_1+\cdots+k_n)}  \arrow{r}{\alpha_{\mathscr A}} & \mathscr A  \\
\mathcal{O}(n) \otimes  \bigotimes_{s=1}^{n}\big( \mathcal{O}(k_s) \otimes \mathscr A^{\otimes k_s}\big) \arrow{r}{id\otimes (\alpha_{\mathscr A}\otimes ...\otimes \alpha_{\mathscr A})} & \mathcal{O}(n) \otimes  \mathscr A^{\otimes n}  \arrow{ur}{\alpha_{\mathscr A}} 
\end{tikzcd}
\]

2) Equivariance: For each $\sigma\in S_n$, we have

\begin{equation*}
\begin{CD}
\mathcal O(n) \otimes \mathscr  A^{\otimes n} @>\sigma\otimes id >> \mathcal O(n) \otimes \mathscr A^{\otimes n} \\
@Vid\otimes \sigma^{-1}VV @V\alpha_{\mathscr A} VV \\
\mathcal O(n) \otimes \mathscr A^{\otimes n} @>\alpha_{\mathscr A} >> \mathscr A\\
\end{CD}
\end{equation*}

\smallskip

3) Unitality: we have
\[ 
\begin{tikzcd}
K \otimes \mathscr A \arrow{r}{\approx} \arrow[swap]{d}{ \eta \otimes id} & \mathscr A \\
\mathcal{O}(1) \otimes \mathscr A  \arrow{ur}{\alpha_{\mathscr A}} 
\end{tikzcd}
\]

We  note that if we put $n =0$,  we obtain a map $\alpha_{\mathscr A}(0) : \mathcal{O}(0) \longrightarrow \mathscr A$ which we denote by $u_{\mathscr A}:\mathcal O(0)\longrightarrow \mathscr A$. 
\end{defn}

Let $\mathcal{O}=\{\mathcal{O}(n)\}_{n\geq 0}$ be an operad and let $\mathscr A$  be an $\mathcal{O}$-algebra.  We now make the convention that for $\theta \in \mathcal{O}(n)$ and any $a_1, \ldots, a_n \in \mathscr A$, we will write 
 \begin{equation}\label{homozg}
 \theta(a_1 \otimes \cdots \otimes a_n):=\alpha_{\mathscr A}(n)(\theta\otimes a_1 \otimes \cdots \otimes a_n)
 \end{equation} As with elements in the Hochschild complex, we will often write $ \theta(a_1 \otimes \cdots \otimes a_n)$ simply as $ \theta(a_1,..., a_n)$. 
If $\mathscr A$, $\mathscr B$ are $\mathcal O$-algebras, we know that an $\mathcal O$-algebra morphism from $\mathscr A$ to $\mathscr B$ is a $K$-linear map $f: \mathscr A \longrightarrow\mathscr  B$ such that, for each $n\geq 0$,  the following diagram commutes:

\begin{equation}\label{homoy}
\begin{CD}
\mathcal{O}(n) \otimes \mathscr A^{\otimes n}@>\alpha_{\mathscr A}(n)>>   \mathscr A  \\
@V{id \otimes f^{\otimes n}}VV @VVfV \\
\mathcal{O}(n) \otimes \mathscr B^{\otimes n} @>\alpha_{\mathscr B}(n)>>\mathscr  B\\
\end{CD}
\end{equation}
Using the convention in \eqref{homozg}, the condition in \eqref{homoy} may be expressed as 
\begin{equation}\label{homox} f(\theta(a_1 \otimes \cdots \otimes a_n) ) = \theta(f(a_1) \otimes \cdots \otimes f(a_n))
\end{equation}
We will denote by $Alg_{\mathcal O}$ the category of $\mathcal O$-algebras. We are now ready to define coalgebra measurings between $\mathcal O$-algebras.

\begin{defn}\label{Defmesopalg} Let $\mathscr A$, $\mathscr B$ be $\mathcal O$-algebras. A measuring $(C,\phi)$ from $\mathscr A$ to $\mathscr B$ consists of a  $K$-coalgebra $(C,\Delta,\epsilon)$ and a  linear map
\begin{equation*}
\phi: C\longrightarrow Hom_K(\mathscr A,\mathscr B)
\end{equation*} which satisfies the conditions
\begin{equation*}\begin{array}{c}\phi(c)\big(\theta(a_1 \otimes \cdots \otimes a_n)\big) = \sum\theta \big(\phi(c_1)(a_1) \otimes \cdots \otimes \phi(c_n)(a_n)\big) \\
\phi(c)\big(u_{\mathscr A}(\theta_0)\big)=\epsilon(c)u_{\mathscr B}(\theta_0)\qquad \forall\textrm{ }\theta_0\in \mathcal O(0)
\\
\end{array}
\end{equation*}
for each $c\in C$, $\theta \in \mathcal{O}(n)$, $n\geq 1$ and $a_1, \ldots, a_n \in \mathscr A$.
\end{defn}

Let  $Coalg_K$ denote the category of $K$-coalgebras. Then,  it is well known (see, for example,  \cite[Theorem 4.1]{Barr}, \cite{Swd}) that  the forgetful functor $Coalg_K\longrightarrow Vect_K$ has a right adjoint $\mathfrak C: Vect_K \longrightarrow Coalg_K$ . In other words, for any $K$-vector space $V$ and any  $K$-coalgebra $C$, there is a natural isomorphism
\begin{equation}\label{adj1} Hom_K(C, V) \cong Coalg_K(C, \mathfrak C(V))
\end{equation} In particular, for any $V\in Vect_K$, there is a canonical morphism $\pi=\pi(V):\mathfrak C(V)\longrightarrow V$ and the 
 pair $(\mathfrak C(V), \pi(V))$ is known as  the cofree coalgebra on $V$.  
 
\begin{Thm}\label{htm5.4}  Let $\mathcal{O}=\{\mathcal{O}(n)\}_{n\geq 0}$ be an operad and let $\mathscr A$ and $\mathscr B$ be $\mathcal{O}$-algebras. Then, there exists a $K$-coalgebra $\mathcal M=\mathcal M(\mathscr A,\mathscr B)$ and a measuring $\Phi : \mathcal M \longrightarrow Hom_K(\mathscr A, \mathscr B)$ satisfying the following property: given any measuring $(C,\phi)$ from $\mathscr A$ to $\mathscr B$, there exists a unique morphism $\psi: C\longrightarrow \mathcal M$ of coalgebras making the following diagram commutative:
\begin{equation}
\begin{tikzcd}[column sep=1.5em]
\mathcal{M}  \arrow[rr, , ""{name=U}, , "\Phi"]{} &&  Hom_K(\mathscr A, \mathscr B) \\
 & C \arrow[from=U, phantom, "\scalebox{1.5}{$\circlearrowright$}" description] \arrow{ur}{\phi} \arrow{ul}{\psi}
\end{tikzcd}
\end{equation}
\end{Thm} 

 \begin{proof}
Let $(\mathfrak C, \pi)$ be the cofree coalgebra on $Hom_K(\mathscr A, \mathscr B)$.  We set $ \mathcal{M}(\mathscr A, \mathscr B):= \sum D$, where the sum is taken over all subcoalgebras $D$  of $\mathfrak C$ such that the restriction $\pi |_D : D \longrightarrow Hom_K(\mathscr A, \mathscr B)$ is a measuring. We note that $ \mathcal{M}(\mathscr A,\mathscr B)$ is a coalgebra and the restriction  $ \Phi = \pi |_{ \mathcal{M}(\mathscr A, \mathscr B)}$ determines a measuring $ (\mathcal{M}(\mathscr A,\mathscr  B), \Phi)$ from $\mathscr A$ to $\mathscr B$.
 
\smallskip

We now consider a measuring $(C, \phi)$ from $\mathscr A$ to $\mathscr B$. Using the adjunction in \eqref{adj1}, the morphism $\phi: C\longrightarrow Hom_K(\mathscr A, \mathscr B)$ corresponds to a unique coalgebra morphism  $\psi : C\longrightarrow \mathfrak C=\mathfrak C(Hom_K(\mathscr A,\mathscr  B))$ such that the following diagram commutes:

\begin{equation*}
\begin{tikzcd}[column sep=1.5em]
\mathfrak C \arrow[rr, , ""{name=U}, , "\pi"]{} && Hom_K(\mathscr A, \mathscr B) \\
 & C\arrow[from=U, phantom, "\scalebox{1.5}{$\circlearrowright$}" description] \arrow{ur}{\phi} \arrow{ul}{\psi}
\end{tikzcd}
\end{equation*}

In other words,  $\pi~ |_{\psi(C)}$ is a measuring and it now follows that the subcoalgebra $\psi(C) \subseteq \mathcal{M}(\mathscr A, \mathscr B)$. Hence,  we have a coalgebra morphism  $\psi:C\longrightarrow \mathcal{M}(\mathscr A, \mathscr B)$ making the following diagram commute. 

\begin{equation*}
\begin{tikzcd}[column sep=1.5em]
\mathcal{M}(\mathscr A, \mathscr B)  \arrow[rr, , ""{name=U}, , "\Phi= \pi |_{ \mathcal{M}(\mathscr A, \mathscr B)}"]{} && Hom_K(\mathscr A, \mathscr B) \\
 & C\arrow[from=U, phantom, "\scalebox{1.5}{$\circlearrowright$}" description] \arrow{ur}{\phi} \arrow{ul}{\psi}
\end{tikzcd}
\end{equation*} 
 
\end{proof}
 
 It is well known that $Coalg_K$ is a monoidal category. It follows therefore (see, for instance, \cite[Chapter 6]{Borc}, \cite{GMK}) that we can consider categories enriched over the monoidal category $Coalg_K$.
Our next aim is to show that the category of $\mathcal{O}$-algebras is enriched over $K$-coalgebras.

\begin{prop}\label{Prp2.3} Let $\mathscr A$, $\mathscr B$ and $\mathscr E$ be $\mathcal{O}$-algebras.  Let $(C, \phi)$ be a measuring from $\mathscr A$ to  $\mathscr B$ and $(D, \psi)$ be a measuring from $\mathscr B$ to $\mathscr E$ . Then, the canonical $K$-linear map
\begin{equation}
\begin{CD}
\varphi:   C \otimes D  @>\phi\otimes\psi>> Hom_K(\mathscr A, \mathscr B) \otimes Hom_K(\mathscr B, \mathscr E)@>\circ>> Hom_K(\mathscr A, \mathscr E)
\end{CD}
\end{equation}
is a measuring  from $\mathscr A$ to $\mathscr E$. In particular, there is an induced coalgebra morphism
$C\otimes D\longrightarrow \mathcal M(\mathscr A, \mathscr E)$.
\end{prop}

 \begin{proof} 
By definition, we have
\begin{equation*}
 \varphi(c \otimes d) = \psi(d) \circ ~ \phi(c) ~~\mbox{for}~~ c \in C, ~~ d \in D.
\end{equation*}

 \smallskip
 
Given coalgebras $(C,\Delta_C,\epsilon_C)$ and $(D,\Delta_D,\epsilon_D)$, we  note that   $(C \otimes D,\Delta,\epsilon)$ is a $K$-coalgebra whose iterated coproducts are given by  $\Delta^{n-1}(c \otimes d)= \sum \sum (c_1 \otimes d_1)\otimes \cdots \otimes (c_n \otimes d_n)$, where $\Delta^{n-1}_C(c)= \sum c_1 \otimes \cdots \otimes c_n$ and $\Delta^{n-1}_D(d)= \sum d_1 \otimes \cdots \otimes d_n$. 
The counit on  $C \otimes D$ is given by $\epsilon (c\otimes d)=\epsilon_C(c)\epsilon_D(d)$. 

\smallskip
We now claim that $\varphi$ is a measuring from $\mathscr A$ to $\mathscr E$.
For this, we consider $\theta\in \mathcal O(n)$, $c\in C$, $d\in D$ and $a_1, \ldots, a_n \in \mathscr A$.  Then, we see that
\begin{align*}
\varphi(c \otimes d)(\theta(a_1 \otimes \cdots \otimes a_n)) &=(\psi(d) \circ  \phi(c))(\theta(a_1 \otimes \cdots \otimes a_n))\\
&=\psi(d) \Big( \sum \theta\big(\phi(c_1)(a_1) \otimes \cdots \otimes \phi(c_n)(a_n)\big)\Big)\\
&=\sum \sum \theta \big((\psi(d_1) \circ \phi(c_1))(a_1) \otimes \cdots \otimes (\psi(d_n) \circ \phi(c_n))(a_n)\big)\\
&= \sum \sum \theta \big( \varphi(c_1 \otimes d_1)(a_1) \otimes \cdots \otimes \varphi(c_n \otimes d_n)(a_n)\big)
\end{align*}
For any $\theta_0\in \mathcal O(0)$, we observe that
\begin{equation*}
\varphi(c\otimes d)(u_{\mathscr A}(\theta_0))=(\psi(d) \circ  \phi(c))(u_{\mathscr A}(\theta_0))=\epsilon_C(c)\epsilon_D(d)u_{\mathscr E}(\theta_0)
\end{equation*}

It follows that $\varphi$  is a measuring from $\mathscr A$ to $\mathscr E$. From the universal property 
 of $\mathcal M(\mathscr A, \mathscr E)$, we now obtain  an induced $K$-coalgebra morphism  $C \otimes D \longrightarrow \mathcal M(\mathscr A,\mathscr  E)$.
\end{proof}

\begin{cor}\label{Cor2.1} Let $\mathscr A$, $\mathscr B$ and $\mathscr C$ be $\mathcal{O}$-algebras.  Then we have a canonical coalgebra morphism 
\begin{equation}
\mathcal{M}(\mathscr A, \mathscr B) \otimes \mathcal{M}( \mathscr B,  \mathscr C) \longrightarrow \mathcal{M}(\mathscr A,\mathscr  C).
\end{equation}
\end{cor}

\begin{proof} This follows directly from Proposition \ref{Prp2.3}. 

\end{proof}

\begin{Thm}\label{Thm2.4} Let $K$ be a field and let $\mathcal{O}=\{\mathcal{O}(n)\}_{n\geq 0}$ be an operad in $Vect_K$. Then, the category of $\mathcal{O}$-algebras is enriched over the category of $K$-coalgebras. 
\end{Thm} 

\begin{proof}  Given  $\mathcal{O}$-algebras $\mathscr A$ and $\mathscr B$, we consider the ``Hom-object'' $\mathcal M(\mathscr A, \mathscr B)$ which lies in $Coalg_K$. We know from 
Corollary \ref{Cor2.1} that given $\mathcal{O}$-algebras $\mathscr A$, $\mathscr B$ and $\mathscr C$,  there is a composition 
\begin{equation}\label{eq2.7}
  \mathcal{M}(\mathscr  A, \mathscr B) \otimes \mathcal{M}(\mathscr  B, \mathscr C) \longrightarrow \mathcal{M}( \mathscr A, \mathscr C).
\end{equation} of Hom-objects in $Coalg_K$. The unit object in $Coalg_K$ is $K$ treated as a coalgebra over itself and it is clear that we have a unit map
\begin{equation}\label{eq2.8}
\phi_{ \mathscr A}: K \longrightarrow Hom_K( \mathscr A,  \mathscr A)
\end{equation} which induces a morphism of coalgebras $ K \longrightarrow \mathcal M(\mathscr A, \mathscr A)$. Together with the composition of Hom-objects in \eqref{eq2.7}, it may now be verified  that the category of $\mathcal{O}$-algebras is enriched over $K$-coalgebras.
\end{proof}

\begin{cor} Let $\mathscr A$ be an $\mathcal{O}$-algebra. Then, $\mathcal M(\mathscr A, \mathscr A)$ carries the structure of a bialgebra.
\end{cor}

\begin{proof}
We already know that  $\mathcal M(\mathscr A, \mathscr A)$ is a coalgebra. From Theorem \ref{Thm2.4}, we know that it is an  endomorphism object of a category enriched over $Coalg_K$. Hence,  $\mathcal M(\mathscr A,\mathscr  A)$ is an algebra object in $Coalg_K$; in other words, it is a bialgebra.
\end{proof}

For the final part of this section, we will assume that $\mathcal O=\mathcal O(E,R)$ is a binary and quadratic operad over an operadic quadratic datum $(E,R)$ (see \cite[$\S$ 7.1.1]{LoVa}) consisting of binary generating operations $E$ and relators $R$ (see \cite[$\S$ 7.1.3]{LoVa}). Then, the operadic homology $HH_\bullet^{\mathcal O}(\mathscr A)$ of an $\mathcal O$-algebra $\mathscr A$ is obtained from the Koszul complex $C_\bullet^{\mathcal O}(\mathscr A)$  (see \cite[$\S$ 12.1.2]{LoVa}) whose terms are given by
\begin{equation}\label{kosz1}
C_n^{\mathcal O}(\mathscr A):=\mathcal O^{\ash}(n+1)\otimes_{S_{n+1}}\mathscr A^{\otimes n+1}
\end{equation} where $\mathcal O^{\ash}$ is the Koszul dual cooperad of $\mathcal O$ (see \cite[$\S$ 7.2.1]{LoVa}). The differential $d:C_n^{\mathcal O}(\mathscr A)
\longrightarrow C_{n-1}^{\mathcal O}(\mathscr A)$ is given by setting   (see \cite[$\S$ 12.1.2]{LoVa}) 
\begin{equation}\label{kosz2}
d(\delta \otimes (a_1,...,a_{n+1})) =\sum \xi\otimes (a_{\sigma^{-1}(1)},...,a_{\sigma^{-1}(i-1)},\mu(a_{\sigma^{-1}(i)},a_{\sigma^{-1}(i+1)}),a_{\sigma^{-1}(i+2)},...,a_{\sigma^{-1}(n+1)})
\end{equation} where $\Delta_{(1)}(\delta)=\sum (\xi;id,...,id,\mu,id,...,id;\sigma)$, $\delta\in \mathcal O^{\ash}(n+1)$, $\xi\in \mathcal O^{\ash}(n)$, $\mu\in \mathcal O^{\ash}(2)=
E$ and $\sigma\in S_{n+1}$. Here $\Delta_{(1)}$ denotes the infinitesimal decomposition map of the cooperad $\mathcal O^{\ash}$ (see \cite[$\S$ 6.1.7]{LoVa}) and
$\mu\in \mathcal O^{\ash}(2)$ is treated as an element of $\mathcal O(2)$ via the twisting morphism $\kappa :\mathcal O^{\ash}\longrightarrow \mathcal O$. 

\begin{Thm}\label{hTh5.9u}  Let $\mathcal O$ be a binary and quadratic operad. Let $\mathscr A$, $\mathscr B$ be $\mathcal O$-algebras and let  $(C_\bullet^{\mathcal O}(\mathscr A), d)$ and $(C_\bullet^{\mathcal O}(\mathscr B), d)$ be their respective Koszul complexes. Let $(C, \phi)$ be a cocommutative measuring from $\mathscr A$ to $\mathscr B$.

\smallskip
  For each $c\in C$, we have a  morphism $\tilde \phi(c):(C_\bullet^{\mathcal O}(\mathscr A), d)\longrightarrow(C_\bullet^{\mathcal O}(\mathscr B), d)$
of complexes given by:
\begin{equation*} 
\begin{array}{c}
\tilde\phi(c)_n:C_n^{\mathcal O}(\mathscr A)\longrightarrow C_n^{\mathcal O}(\mathscr B)\\  \tilde\phi(c)_n(\delta \otimes (a_1,...,a_{n+1})):= \delta \otimes (\sum  (\phi(c_1)(a_1),....,\phi(c_{n+1})(a_{n+1}))) \\
\end{array}
\end{equation*} In particular, we have an induced morphism $\tilde \phi(c): HH_\bullet^{\mathcal O}(\mathscr A)\longrightarrow HH_\bullet^{\mathcal O}(\mathscr B)$ on operadic
homologies for each $c\in C$. 

\end{Thm}

\begin{proof} We maintain the notation from \eqref{kosz1} and \eqref{kosz2}. Since $C$ is cocommutative, the maps $\tilde\phi(c)_n:C_n^{\mathcal O}(\mathscr A)\longrightarrow C_n^{\mathcal O}(\mathscr B)$ on the terms in the Koszul complexes are well defined. 

\smallskip We have $\Delta_{(1)}(\delta)=\sum (\xi;id,...,id,\mu,id,...,id;\sigma)$, $\delta\in \mathcal O^{\ash}(n+1)$, $\xi\in \mathcal O^{\ash}(n)$, $\mu\in \mathcal O^{\ash}(2)$ and $\sigma\in S_{n+1}$. From the definitions and the fact that $(C,\phi)$ is a measuring, we obtain
\begin{equation*}
\begin{array}{l}
\tilde\phi(c)_{n-1}(d(\delta \otimes (a_1,...,a_{n+1})))\\ =\sum\tilde\phi(c)_{n-1}( \xi\otimes (a_{\sigma^{-1}(1)},...,a_{\sigma^{-1}(i-1)},\mu(a_{\sigma^{-1}(i)},a_{\sigma^{-1}(i+1)}),a_{\sigma^{-1}(i+2)},...,a_{\sigma^{-1}(n+1)})) \\
=\sum  \xi\otimes (\phi(c_1)(a_{\sigma^{-1}(1)}),... ,\phi(c_{i})(\mu(a_{\sigma^{-1}(i)},a_{\sigma^{-1}(i+1)})), ...,\phi(c_{n})(a_{\sigma^{-1}(n+1)})) \\
=\sum  \xi\otimes (\phi(c_1)(a_{\sigma^{-1}(1)}),... , \mu(\phi(c_{i})(a_{\sigma^{-1}(i)}),\phi(c_{i+1})(a_{\sigma^{-1}(i+1)})), ...,\phi(c_{n+1})(a_{\sigma^{-1}(n+1)})) \\
\end{array}
\end{equation*} On the other hand, since $C$ is cocommutative, we have
\begin{equation*}
\begin{array}{l}
d(\tilde\phi(c)_n(\delta \otimes (a_1,...,a_{n+1})))=d(\delta \otimes (\sum  (\phi(c_1)(a_1),....,\phi(c_{n+1})(a_{n+1})))) \\
=\sum \xi \otimes (\phi(c_{\sigma^{-1}(1)})(a_{\sigma^{-1}(1)}),... , \mu(\phi(c_{\sigma^{-1}(i)})(a_{\sigma^{-1}(i)}),\phi(c_{\sigma^{-1}(i+1)})(a_{\sigma^{-1}(i+1)})), ...,\phi(c_{
\sigma^{-1}(n+1)})(a_{\sigma^{-1}(n+1)}))\\
=\sum  \xi\otimes (\phi(c_1)(a_{\sigma^{-1}(1)}),... , \mu(\phi(c_{i})(a_{\sigma^{-1}(i)}),\phi(c_{i+1})(a_{\sigma^{-1}(i+1)})), ...,\phi(c_{n+1})(a_{\sigma^{-1}(n+1)})) \\
\end{array}
\end{equation*} In other words, the maps $\{\tilde\phi(c)\}_{n\geq 0}$ commute with the differentials on the Koszul complexes. This proves the result.

\end{proof}

\section{Measurings of $\mathcal O$-algebras and their universal enveloping algebras}

In Section 4, we showed that a measuring of Lie algebras induces a measuring of their universal enveloping algebras. In general, if $\mathscr A$ is an $\mathcal O$-algebra, there is an ordinary associative
algebra $U_{\mathcal O}(\mathscr A)$ (see \cite[$\S$ 1.6]{GK}, \cite{HS93}) such that the category of left modules over $U_{\mathcal O}(\mathscr A)$ is equivalent to that of modules over the 
$\mathcal O$-algebra $\mathscr A$. In order to describe this algebra, we need to recall the following operation (see, for instance, \cite[Proposition 13]{Markl}): for $m$, $n\geq 0$ and $1\leq i\leq m$, we have
\begin{equation}\label{circi}
\begin{CD}
\circ_i: \mathcal O(m)\otimes \mathcal O(n)@>\cong >> \mathcal O(m)\otimes K^{\otimes i-1}\otimes \mathcal O(n)\otimes K^{\otimes m-i} @. \\
@. @VVid\otimes \eta^{\otimes i-1}
\otimes id\otimes \eta^{\otimes m-i}V\\
@.  \mathcal O(m)\otimes (\mathcal O(1)^{\otimes i-1}\otimes \mathcal O(n)\otimes \mathcal O(1)^{\otimes m-i}) @>\gamma>>
 \mathcal O(m+n-1)\\
 \end{CD}
\end{equation} In particular, if $\mathscr A$ is an $\mathcal O$-algebra, we have
\begin{equation}\label{circA}
(\theta\circ_i\theta')(a_1,...,a_{m+n-1}):=\theta(a_1,...,a_{i-1},\theta'(a_i,...,a_{i+n-1}),a_{i+n},...,a_{m+n-1})
\end{equation} for $\theta\in \mathcal O(m)$, $\theta'\in \mathcal O(n)$ and $a_1\otimes ... \otimes a_{m+n-1}\in \mathscr A^{\otimes m+n-1}$.

\smallskip

If $\mathscr A$ is an $\mathcal O$-algebra, we now recall 
(see \cite[$\S$ 1.6]{GK}) that its universal enveloping algebra $U_{\mathcal O}(\mathscr A)$ is generated (as a vector space) by symbols $Z(\theta; a_1,...,a_{m})$, $\theta\in \mathcal O(m+1)$, 
$a_1$,...,$a_{m}\in \mathscr A$, $m\geq 0$ subject to polylinearity in each variable and the following relations
\begin{equation}\label{gkrel}
\begin{array}{c}
Z(\sigma\cdot \theta; a_1,...,a_{m})=Z(\theta; a_{\sigma(1)},...,a_{\sigma(m)})\\
Z(\theta\circ_i\theta'; a_1,...,a_{m+n})=Z(\theta; a_1,...,a_{i-1},\theta'(a_{i},...,a_{i+n}),a_{i+n+1},...,a_{m+n})\\
\end{array}
\end{equation}
for each $1\leq i\leq m$ and $\sigma\in S_m$, where $\theta\in \mathcal O(m+1)$,  $\theta'\in \mathcal O(n+1)$, $( a_1,...,a_{m+n})\in \mathscr A^{\otimes m+n}$. Here, $\sigma\in S_m$
acts on $\theta\in \mathcal O(m+1)$ via the inclusion $S_m\hookrightarrow S_{m+1}$. 

\smallskip The multiplication in $U_{\mathcal O}(\mathscr A)$ is determined by
\begin{equation}\label{mgkrel}
Z(\theta; a_1,...,a_{m})\cdot Z(\mu; b_1,...,b_{l})=Z(\theta\circ_{m+1}\mu;a_1,...,a_{m},b_1,...,b_{l})
\end{equation} for $\theta\in \mathcal O(m+1)$, $\mu\in \mathcal O(l+1)$ and $a_i$, $b_j\in\mathscr  A$.

\begin{lem}\label{Lema6.9} Let $\mathscr A$, $\mathscr B$ be $\mathcal O$-algebras and let $(C,\phi)$ be a  cocommutative  measuring from $\mathscr A$ to $\mathscr B$. Then, there is an induced
linear map
\begin{equation}
U_{\mathcal O}(\phi):C\longrightarrow Hom_K(U_{\mathcal O}(\mathscr A),U_{\mathcal O}(\mathscr B))
\end{equation}

\end{lem} 

\begin{proof} Maintaining the notation above, we set
\begin{equation}\label{eq6.15rx}
U_{\mathcal O}(\phi)(c)(Z(\theta; a_1,...,a_{m})):=Z(\theta,\phi(c_1)(a_1),...,\phi(c_{m})(a_m))=Z(\theta;c_1\cdot a_1,...,c_m\cdot a_m)
\end{equation} for each $c\in C$. It is evident that this definition is well behaved with respect to polylinearity of the symbols $Z(\theta; a_1,...,a_{m})$ in each variable. Since  $\phi:C
\longrightarrow Hom_K(\mathscr A,\mathscr B)$ is a measuring of $\mathcal O$-algebras, we also have
\begin{equation*}
\begin{array}{l}
U_{\mathcal O}(\phi)(c)(Z(\theta\circ_i\theta'; a_1,...,a_{m+n}))\\=Z(\theta\circ_i\theta',c_1\cdot a_1,...,c_{m+n}\cdot a_{m+n})\\
=Z(\theta; c_1\cdot a_1,...,c_{i-1}\cdot a_{i-1},\theta'(c_i\cdot a_{i},...,c_{i+n}\cdot a_{i+n}),c_{i+n+1}\cdot a_{i+n+1},...,c_{m+n}\cdot a_{m+n})\\
=Z(\theta;c_1\cdot  a_1,...,c_{i-1}\cdot a_{i-1},c_i\cdot \theta'(a_{i},...,a_{i+n}),c_{i+1}\cdot a_{i+n+1},...,c_{m}\cdot a_{m+n})\\
=U_{\mathcal O}(\phi)(c)(Z(\theta; a_1,...,a_{i-1},\theta'(a_{i},...,a_{i+n}),a_{i+n+1},...,a_{m+n}))\\
\end{array}
\end{equation*} Additionally, since $C$ is cocommutative, we see that
\begin{equation*}
\begin{array}{ll}
U_{\mathcal O}(\phi)(c)(Z(\sigma\cdot \theta; a_1,...,a_{m}))&=Z(\sigma\cdot \theta; c_1\cdot a_1,...,c_m\cdot a_m)\\
&=Z(\theta; c_{\sigma(1)}\cdot a_{\sigma(1)},...,c_{\sigma(m)}\cdot a_{\sigma(m)})\\
&=Z(\theta; c_{1}\cdot a_{\sigma(1)},...,c_{m}\cdot a_{\sigma(m)})\\
&=U_{\mathcal O}(\phi)(c)(Z(\theta; a_{\sigma(1)},...,a_{\sigma(m)})\\ 
\end{array}
\end{equation*} Hence, the definition in \eqref{eq6.15rx} is well behaved with respect to both relations in \eqref{gkrel}. This proves the result.

\end{proof}

\begin{Thm}\label{rP6.9} Let $\mathscr A$, $\mathscr B$ be $\mathcal O$-algebras and let $(C,\phi)$ be a  cocommutative  measuring from $\mathscr A$ to $\mathscr B$. Then, the  induced
linear map
\begin{equation}
U_{\mathcal O}(\phi):C\longrightarrow Hom_K(U_{\mathcal O}(\mathscr A),U_{\mathcal O}(\mathscr B))
\end{equation} is a measuring of algebras from $U_{\mathcal O}(\mathscr A)$ to $U_{\mathcal O}(\mathscr B)$.
\end{Thm} 

\begin{proof}
We maintain the notation from \eqref{mgkrel}. We observe that for any $c\in C$, we have
\begin{equation*}
\begin{array}{ll}
U_{\mathcal O}(\phi)(c)(Z(\theta\circ_{m+1}\mu;a_1,...,a_{m},b_1,...,b_{l}))&=\sum  Z(\theta\circ_{m+1}\mu;c_1\cdot a_1,...,c_m\cdot a_{m},c_{m+1}\cdot b_1,...,c_{m+l}\cdot b_{l})\\
&=\sum Z(\theta; c_1\cdot a_1,...,c_m\cdot a_{m})\cdot Z(\mu; c_{m+1}\cdot b_1,...,c_{m+l}\cdot b_{l})  \\
&=\sum U_{\mathcal O}(c_1)(Z(\theta; a_1,...,a_{m}))\cdot U_{\mathcal O}(\phi)(c_2)( Z(\mu; b_1,...,b_{l})) \\
\end{array}
\end{equation*} This proves the result.
\end{proof}

\begin{cor} Let $\mathscr A$, $\mathscr B$ be $\mathcal O$-algebras. Then, there is a canonical morphism of coalgebras $\mathcal M_{\mathcal O}(\mathscr A,\mathscr B)\longrightarrow \mathcal M(U_{\mathcal O}(\mathscr A),
U_{\mathcal O}(\mathscr B))$. 
\end{cor}

\begin{proof} Using Theorem \ref{rP6.9}, it follows that the universal measuring $\mathcal M_{\mathcal O}(\mathscr A,\mathscr B)\longrightarrow Hom_K(\mathscr A,\mathscr B)$ induces a measuring
$\mathcal M_{\mathcal O}(\mathscr A,\mathscr B)\longrightarrow Hom_K(U_{\mathcal O}(\mathscr A),U_{\mathcal O}(\mathscr B))$ of ordinary algebras. The result is now clear from the universal property of the coalgebra
$\mathcal M(U_{\mathcal O}(\mathscr A),U_{\mathcal O}(\mathscr B))$. 

\end{proof}

\section{Measuring comodules for algebras over operads} 

We continue with $\mathcal O=\{\mathcal O(n)\}_{n\geq 0}$ being an operad of vector spaces over $K$.  We first recall the well known notion of modules over $\mathcal O$-algebras.

\begin{defn}\label{oalgmod} (see, for instance, \cite[Definition 3.4]{Baues} or \cite[Chapter 12]{LoVa})
Let $\mathcal{O}=\{\mathcal{O}(n)\}_{n\geq 0}$ be an operad and let $\mathscr A$ be an $\mathcal{O}$-algebra. An  $\mathscr A$-module $\mathscr M$ is a $K$-vector space $\mathscr M$ together with  maps:
\begin{equation}\label{modop}\alpha_{\mathscr M}(n): \mathcal{O}(n) \otimes \mathscr A^{\otimes n-1} \otimes \mathscr  M\longrightarrow \mathscr M, ~~ \mbox{for all}~~ n \geq 1
\end{equation} satisfying the following conditions.

\smallskip
1) Associativity:  For  $n,k_1, \ldots, k_n\geq 0$ and $k=\sum k_i$,  the following diagram commutes

\[
\begin{tikzcd}[column sep=large]
\big(\mathcal{O}(n) \otimes \bigotimes_{s=1}^{n} \mathcal{O}(k_s) \big)   \bigotimes\mathscr A^{k-1}\otimes \mathscr M \arrow{r}{\gamma \otimes id\otimes id} \arrow[swap]{d}{\textit{Shuffle}} & \mathcal{O}(k) \otimes  \mathscr A^{\otimes k-1} \otimes\mathscr M \arrow{r}{\alpha_{\mathscr M}} & \mathscr M  \\
\mathcal{O}(n) \otimes  \bigotimes_{s=1}^{n-1}\big( \mathcal{O}(k_s) \otimes \mathscr A^{\otimes k_s}\big) \otimes \big(\mathcal O(k_n)\otimes \mathscr A^{\otimes k_n-1}
\otimes \mathscr M\big)\textrm{ }\arrow{r}{id\otimes \alpha_{\mathscr A}^{\otimes n-1}\otimes  \alpha_{\mathscr M}} \quad & \quad \mathcal{O}(n) \otimes  \mathscr A^{\otimes n-1}\otimes \mathscr M  \arrow{ur}{\alpha_{\mathscr M}} 
\end{tikzcd}
\]

2) Equivariance: For each $\sigma\in S_{n-1}$, the following diagram commutes:

\begin{equation*}
\begin{CD}
\mathcal O(n) \otimes \mathscr  A^{\otimes n-1}\otimes \mathscr M @>\sigma\otimes id >> \mathcal O(n) \otimes \mathscr A^{\otimes n-1}\otimes \mathscr M \\
@Vid\otimes \sigma^{-1}VV @V\alpha_{\mathscr M} VV \\
\mathcal O(n) \otimes \mathscr A^{\otimes n-1} \otimes \mathscr M @>\alpha_{\mathscr M} >> \mathscr M\\
\end{CD}
\end{equation*} Here, $\sigma\in S_{n-1}$ acts on $\mathcal O(n)$ and $\mathscr  A^{\otimes n-1}\otimes \mathscr M$ via the inclusion $S_{n-1}\hookrightarrow S_n$.

\smallskip

3) Unitality:  The following diagram commutes
\[ 
\begin{tikzcd}
K \otimes \mathscr M \arrow{r}{\approx} \arrow[swap]{d}{ \eta \otimes id} & \mathscr M \\
\mathcal{O}(1) \otimes \mathscr M  \arrow{ur}{\alpha_{\mathscr M}} 
\end{tikzcd}
\]

\end{defn}

Let $\mathscr A$ be an $\mathcal O$-algebra and let $\mathscr M$ be an $\mathscr A$-module. We now make the convention that for $\theta \in \mathcal{O}(n)$, $m \in \mathscr M$  and any $a_1, \ldots, a_{n-1} \in \mathscr A$, we will write 
 \begin{equation}\label{homoz}
 \theta_{\mathscr  M}(a_1 \otimes \cdots \otimes a_{n-1}\otimes m):=\alpha_{\mathscr  M}(n)(\theta\otimes a_1 \otimes \cdots \otimes a_{n-1}\otimes m)
 \end{equation} 
Let $(C, \Delta, \epsilon)$ be a $K$-coalgebra and 
 let $(P,\omega)$ be a left comodule over $C$ with structure map $\omega:  P \longrightarrow C \otimes P$. Then, we know    that $(id \otimes \omega) \circ \omega = (\Delta \otimes id) \circ \omega$ and $(\epsilon \otimes id) \circ \omega = id$. For $p \in P$,  we write $\omega(p) = \sum p_0 \otimes p_1$. We are now ready to define  measuring comodules over a measuring coalgebra $(C,\phi)$.

\begin{defn}\label{Defmescomod}
Let $\mathscr A$, $\mathscr B$ be  $\mathcal{O}$-algebras and let $(C, \phi)$ be a measuring from $\mathscr A$ to $\mathscr B$. Let $\mathscr M$ be an $\mathscr A$-module and $\mathscr N$  a $\mathscr B$ module. A (left) measuring comodule over $(C, \phi)$ from $\mathscr M$ to $\mathscr N$ consists of a left $C$-comodule $P$ and a $K$-linear map $ \psi: P \longrightarrow Hom_K(\mathscr M,\mathscr  N)$ satisfying 
\begin{equation}\label{meascomod}
\psi(p)\big(\theta_{\mathscr  M} (a_1 \otimes \cdots \otimes a_{n-1}\otimes m)\big)= \sum \theta_{\mathscr  N }\big( \phi(p_0)(a_1) \otimes \cdots \otimes \phi(p_{n-2})(a_{n-1}) \otimes \psi(p_{n-1})(m)\big)
\end{equation}
for all $\theta\in \mathcal O(n)$, $p \in P$, $a_1, \ldots, a_{n-1} \in \mathscr A$ and $ m \in \mathscr M$.
\end{defn}

\smallskip

We will now show that there is a universal measuring comodule over $(C, \phi)$ from $\mathscr M$ to $\mathscr N$.  Since $K$ is a field, we know that $C$ is flat over $K$. It follows that the category $C-Comod$ 
of left $C$-comodules is a Grothendieck category (see, for instance, \cite[$\S$ 1]{Wis}). From general properties of Grothendieck categories (see, for instance, \cite[Proposition 8.3.27]{KaSc}){, it now follows that  the forgetful functor $C-Comod \longrightarrow Vect_K$ must have a right adjoint, which we denote by $\mathfrak R_C: Vect_K \longrightarrow C-Comod$. In other words, for any $K$-vector space $V$ and any  $C$-comodule $P$, there is a natural isomorphism
\begin{equation}\label{adj2} Hom_K(P, V) \cong  C-Comod(P, \mathfrak R_C(V))
\end{equation}
 In particular, for any $V\in Vect_K$, there is a canonical morphism $\Lambda=\Lambda_C(V):\mathfrak R_C(V) \longrightarrow V$ of vector spaces.

\begin{Thm}\label{Th3.2}  Let $\mathscr A$ and $\mathscr B$ be $\mathcal{O}$-algebras. Let $\mathscr M$ be an $\mathscr A$-module, $\mathscr N$ be a $\mathscr B$-module and let $(C, \phi)$ be a measuring from $\mathscr A$ to $\mathscr  B$. Then, there
exists  a measuring $(C, \phi)$-comodule $(\mathcal Q_C=\mathcal Q_C(\mathscr M,\mathscr  N), \Psi: \mathcal Q_C \longrightarrow Hom_K(\mathscr M, \mathscr N))$ satisfying the following property: given any measuring  $(C, \phi)$-comodule $(P,\psi)$ from $\mathscr M$ to $\mathscr N$, there exists a unique morphism $\xi: P \longrightarrow \mathcal Q_C$ of $C$-comodules making the following diagram commutative.
\begin{equation*}
\begin{tikzcd}[column sep=1.5em]
\mathcal{Q}_C   \arrow[rr, , ""{name=U}, , "\Psi"]{} &&  Hom_K(\mathscr M, \mathscr N) \\
 & P \arrow[from=U, phantom, "\scalebox{1.5}{$\circlearrowright$}" description] \arrow{ur}{\psi} \arrow{ul}{\xi }
\end{tikzcd}
\end{equation*}
\end{Thm} 

\begin{proof} We put $V=Hom_K(\mathscr M,\mathscr  N)$ and consider the left $C$-comodule $\mathfrak R_C:=\mathfrak R_C(V)$ along with the canonical morphism
$\Lambda =\Lambda_C(V):\mathfrak R_C(V) \longrightarrow V$. 
 We set $ \mathcal{Q}_C(\mathscr M, \mathscr N):= \sum Q$, where the sum is taken over all left $C$-subcomodules $Q$  of $\mathfrak R_C$ such that the restriction $\Lambda|_Q : Q \longrightarrow Hom_K(\mathscr M, \mathscr N)$ is a measuring comodule over $(C,\phi)$. We note that $ \mathcal{Q}_C(\mathscr M, \mathscr N)$ is a $C$-comodule and the restriction  $ \Psi = \Lambda |_{ \mathcal{Q}_C(\mathscr M, \mathscr N)}$ determines a measuring  comodule $ (\mathcal{Q}_C(\mathscr M, \mathscr N), \Psi)$ over $(C,\phi)$.

\smallskip
We now consider a measuring $(C, \phi)$-comodule $(P, \psi)$ from $\mathscr M$ to $\mathscr N$. Using the adjunction in \eqref{adj2}, the morphism $\psi: P \longrightarrow Hom_K(\mathscr M, \mathscr N)$ corresponds to a unique $C$-comodule morphism  $\xi : P\longrightarrow \mathfrak R_C=\mathfrak R_C(Hom_K(\mathscr M, \mathscr N))$ such that the following diagram commutes:

\begin{equation*}
\begin{tikzcd}[column sep=1.5em]
\mathfrak R_C  \arrow[rr, , ""{name=U}, , "\Lambda"]{} && Hom_K(\mathscr M, \mathscr N) \\
 & P\arrow[from=U, phantom, "\scalebox{1.5}{$\circlearrowright$}" description] \arrow{ur}{\psi} \arrow{ul}{\xi}
\end{tikzcd}
\end{equation*}

In other words,  $\Lambda~ |_{\xi(P)}$ is a measuring $(C, \phi)$-comodule  and it now follows that the subcomodule $\xi(P) \subseteq \mathcal{Q}_C(\mathscr M,\mathscr  N)$. Hence,  we have a comodule morphism  $\xi: P \longrightarrow \mathcal{Q}_C(\mathscr M, \mathscr N)$ making the following diagram commute. 

\begin{equation*}
\begin{tikzcd}[column sep=1.5em]
\mathcal{Q}_C(\mathscr M,\mathscr  N)  \arrow[rr, , ""{name=U}, , "\Psi"]{} && Hom_K(\mathscr M,\mathscr  N) \\
 & P \arrow[from=U, phantom, "\scalebox{1.5}{$\circlearrowright$}" description] \arrow{ur}{\psi} \arrow{ul}{\xi}
\end{tikzcd}
\end{equation*}
\end{proof}

We will now prove a composition result for measuring comodules.

\begin{prop}\label{P3.3} Let $\mathscr  A$, $\mathscr  B$ and $\mathscr  E$ be $\mathcal{O}$-algebras and let $\mathscr  M$, $\mathscr  N$ and $\mathscr  L$ be modules over $\mathscr  A$, $\mathscr  B$ and $\mathscr  E$ respectively. Suppose that we are given the following:

\smallskip

(1) A measuring $(C,\phi)$ from $ \mathscr  A$ to $\mathscr   B$ and a measuring $(D,\phi')$ from $\mathscr  B$ to $\mathscr  E$.

\smallskip

(2) A measuring $(C,\phi)$-comodule $(P, \psi)$ with $\psi: P\longrightarrow Hom_K(\mathscr  M, \mathscr  N)$.

\smallskip

(3) A measuring $(D,\phi')$-comodule $(Q, \psi')$ with $\psi':Q\longrightarrow Hom_K(\mathscr  N,\mathscr   L)$. 

\smallskip
Then, the induced morphism
\begin{equation*}
\tau : P\otimes Q\xrightarrow{\psi\otimes \psi'} Hom_K(\mathscr  M, \mathscr  N)\otimes Hom_K(\mathscr  N, \mathscr  L)\overset{\circ}{\longrightarrow }
Hom_K(\mathscr  M,\mathscr   L)
\end{equation*} is a measuring $C\otimes D$-comodule.

\end{prop}

\begin{proof}
From Proposition \ref{Prp2.3}, we know that there is a measuring $\varphi: C\otimes D\longrightarrow Hom_K(\mathscr  A,\mathscr   E)$. We now consider an element $p\otimes q\in P\otimes Q$,  along with  $\theta\in \mathcal O(n)$, $a_1, \cdots, a_{n-1} \in \mathscr  A$ and $m\in \mathscr  M$. Then, we have
\begin{equation*}
\begin{array}{ll}
\tau(p\otimes q)(\theta_{\mathscr  M}(a_1 \otimes \cdots \otimes a_{n-1}\otimes m))  & = \psi'(q) \circ \psi(p) \big(\theta_{\mathscr  M}(a_1 \otimes \cdots \otimes a_{n-1}\otimes m)\big)\\
& =\psi'(q)\Big( \sum \theta_{\mathscr  N }\big( \phi(p_0)(a_1) \otimes \cdots \otimes \phi(p_{n-2})(a_{n-1}) \otimes \psi(p_{n-1})(m)\big)\Big) \\
& =\sum \psi'(q)\Big( \theta_{\mathscr  N }\big( \phi(p_0)(a_1) \otimes \cdots \otimes \phi(p_{n-2})(a_{n-1}) \otimes \psi(p_{n-1})(m)\big)\Big) \\
&=\sum \sum \theta_{\mathscr  L}\Big(\phi'(q_0)\big(\phi(p_0)(a_1)\big) \otimes \cdots \otimes \psi'(q_{n-1})\big(\psi(p_{n-1})(m)\big)\Big) \\
&=\sum \sum \theta_{\mathscr  L}\Big(\varphi(p_0 \otimes q_0)(a_1) \otimes \cdots  \otimes \tau(p_{n-1}\otimes q_{n-1})(m)\Big) \\
\end{array}
\end{equation*} This proves the result. 
\end{proof}

\begin{Thm}\label{T3.4} Let $\mathscr M$, $\mathscr N$ and $\mathscr L$ be modules over $\mathcal{O}$-algebras $\mathscr A$, $\mathscr B$
and $\mathscr E$ respectively. Suppose that we are given a measuring $(C,\phi)$ from $\mathscr   A$ to $ \mathscr  B$ and a measuring $(D,\phi')$ from $\mathscr  B$ to $\mathscr  E$.

\smallskip

Then, there is a canonical morphism of left $(C\otimes D)$-comodules:
\begin{equation*}
\mathcal Q_C(\mathscr  M, \mathscr  N)\otimes \mathcal Q_D(\mathscr  N, \mathscr  L)\longrightarrow \mathcal Q_{C\otimes D}(\mathscr  M, \mathscr  L)
\end{equation*} 

\end{Thm}

\begin{proof}
Applying Proposition \ref{P3.3} to the measuring $(C,\phi)$-comodule $\mathcal Q_C(\mathscr  M, \mathscr  N)$ and the measuring
$(D,\phi')$-comodule $\mathcal Q_D(\mathscr  N, \mathscr  L)$, we obtain a measuring $C\otimes D$-comodule $\mathcal Q_C(\mathscr  M, \mathscr  N)\otimes \mathcal Q_D(\mathscr  N, \mathscr  L)$. The universal property of $ \mathcal Q_{C\otimes D}(\mathscr  M, \mathscr  L)$ from 
Theorem \ref{Th3.2} now gives a $(C\otimes D)$-comodule morphism $\mathcal Q_C(\mathscr  M ,\mathscr   N)\otimes \mathcal Q_D(\mathscr  N,\mathscr  L)\longrightarrow \mathcal Q_{C\otimes D}(\mathscr  M, \mathscr  L)$. This proves the result.

\end{proof}

In Section 6, we mentioned (see \cite[$\S$ 1.6]{GK}) that modules over an $\mathcal O$-algebra $\mathscr A$ correspond to left modules over the universal  enveloping
algebra $U_{\mathcal O}(\mathscr A)$. This may be made explicit as follows : if $\mathscr M$ is an $\mathscr A$-module in the sense of Definition \ref{oalgmod}, the corresponding
left module $U_{\mathscr A}(\mathscr M)$ over $U_{\mathcal O}(\mathscr A)$ is identical to $\mathscr M$ as a vector space, with left $U_{\mathcal O}(\mathscr A)$-action defined by setting
\begin{equation}\label{uam}
Z(\theta; a_1,...,a_m)\cdot x= \theta_{\mathscr M}(a_1\otimes ...\otimes a_m\otimes x)=\alpha_{\mathscr M}(m+1)(\theta\otimes a_1\otimes ...\otimes a_m\otimes x)
\end{equation} for $\theta\in \mathcal O(m+1)$, $a_1\otimes ...\otimes a_m\in \mathscr A^{\otimes m}$ and $x\in \mathscr M$. 
We will now show that measurings of modules over an $\mathcal O$-algebra $\mathscr A$ correspond to measurings of modules
over $U_{\mathcal O}(\mathscr A)$. 

\begin{prop}\label{P7.6or} Let $\mathscr  A$, $\mathscr  B$ be $\mathcal O$-algebras. Let $C$ be a cocommutative coalgebra. Let $\phi:C\longrightarrow
Hom_K(\mathscr  A,\mathscr  B)$ be a measuring and $U_{\mathcal O}(\phi):C\longrightarrow Hom_K(U_{\mathcal O}(\mathscr  A),U_{\mathcal O}(\mathscr  B))$ the corresponding measuring of associative algebras. 

\smallskip
Let $\mathscr  M$ be an $\mathscr  A$-module and $\mathscr  N$ be a $\mathscr  B$-module. Let $P$ be a left $C$-comodule and $\psi:P\longrightarrow Hom_K(\mathscr  M,\mathscr  N)$ a morphism.  Then,
$(P,\psi)$ is a measuring comodule over $(C,\phi:C\longrightarrow
Hom_K(\mathscr  A,\mathscr  B))$ if and only if $(P,\psi)$ is  a measuring comodule over $(C,U_{\mathcal O}(\phi):C\longrightarrow Hom_K(U_{\mathcal O}(\mathscr  A),U_{\mathcal O}(\mathscr  B)))$. 
\end{prop} 

\begin{proof} We begin by supposing that $(P,\psi)$ is a measuring comodule over $(C,\phi:C\longrightarrow
Hom_K(\mathscr  A,\mathscr  B))$ from $\mathscr  M$ to $\mathscr  N$. We consider elements $Z(\theta;a_1,...,a_m)\in U_{\mathcal O}(\mathscr  A)$, $x\in M$ and $p\in P$. Then, we have
\begin{equation*}
\begin{array}{ll}
\psi(p)(Z(\theta;a_1,...,a_m)\cdot x)&=\psi(p)(\theta_{\mathscr  M}(a_1\otimes ...\otimes a_m\otimes x))\\
&= \sum \theta_{\mathscr  N }\big( \phi(p_0)(a_1) \otimes \cdots \otimes \phi(p_{m-1})(a_{m}) \otimes \psi(p_{m})(x)\big)\\
&\\
\sum ( U_{\mathcal O}(\phi)(p_0)(Z(\theta;a_1,...,a_m)))\cdot (\psi(p_1)(x)) & =\sum Z(\theta;\phi(p_0)(a_1)\otimes...\otimes \phi(p_{m-1})(a_{m}))\cdot (\psi(p_{m})(x))\\
&=  \sum \theta_{\mathscr  N} \big( \phi(p_0)(a_1) \otimes \cdots \otimes \phi(p_{m-1})(a_{m}) \otimes \psi(p_{m})(x)\big)\\
\end{array}
\end{equation*} Conversely, suppose that $(P,\psi)$ is a measuring comodule over   $(C,U_{\mathcal O}(\phi):C\longrightarrow Hom_K(U_{\mathcal O}(\mathscr  A),U_{\mathcal O}(\mathscr  B))))$ from $U_{\mathscr  A}(\mathscr  M)$
to $U_{\mathscr  B}(\mathscr  N)$.  We consider elements
$\theta\in \mathcal O(m+1)$, $p \in P$, $a_1, \ldots, a_{m} \in \mathscr A$ and $ x \in \mathscr M$. Then, we have
\begin{equation*}
\begin{array}{ll}
\psi(p)\big(\theta_{\mathscr M} (a_1 \otimes \cdots \otimes a_{m}\otimes x)\big) & =\psi(p)(Z(\theta; a_1,...,a_m)\cdot x)\\
&=\sum U_{\mathcal O}(\phi)(p_0)(Z(\theta; a_1,...,a_m))\cdot \psi(p_1)(x)\\
&=\sum Z(\theta;\phi(p_0)(a_1)\otimes ...\otimes \phi(p_{m-1})(a_{m}))\cdot \psi(p_{m})(x) \\
&=\sum \theta_{\mathscr N}(\phi(p_0)(a_1)\otimes ...\otimes \phi(p_{m-1})(a_{m})\otimes \psi(p_{m})(x))
\end{array}
\end{equation*} This proves the result.

\end{proof}

In the setup of Proposition \ref{P7.6or}, $U_{\mathscr  A}(\mathscr  M)$ is an ordinary left module over the associative algebra $U_{\mathcal O}(\mathscr  A)$ and  $U_{\mathscr  B}(\mathscr  N)$ is an ordinary left module over the associative algebra $U_{\mathcal O}(\mathscr  B)$. Following Batchelor \cite{Bat}, we may therefore consider a universal measuring comodule $(Q_C(U_{\mathscr  A}(\mathscr  M),U_B(N)),\psi:Q_C(U_{\mathscr  A}(\mathscr  M),U_{\mathscr  B}(\mathscr  N))
\longrightarrow Hom_K(U_{\mathscr  A}(\mathscr  M),U_{\mathscr  B}(\mathscr  N)))$ over the measuring $(C,U_{\mathcal O}(\phi))$ of associative algebras. 

\begin{Thm}\label{T7.7}  Let $\mathscr  A$, $\mathscr  B$ be $\mathcal O$-algebras. Let $C$ be a cocommutative coalgebra. Let $\phi:C\longrightarrow
Hom_K(\mathscr A,\mathscr B)$ be a measuring and $U_{\mathcal O}(\phi):C\longrightarrow Hom_K(U_{\mathcal O}(\mathscr  A),U_{\mathcal O}(\mathscr  B)))$ the corresponding measuring of associative algebras. 

\smallskip
Let $\mathscr M$ be an $\mathscr A$-module and $\mathscr N$ be a $\mathscr B$-module. Then, we have an isomorphism of $C$-comodules
\begin{equation}
\mathcal Q_C(\mathscr M,\mathscr  N)=Q_C(U_{\mathscr A}(\mathscr M),U_{\mathscr B}(\mathscr N))
\end{equation} 

\end{Thm} 

\begin{proof} This is clear from Proposition \ref{P7.6or} and the universal property in Theorem \ref{Th3.2}. 
\end{proof} 

\section{Universal measurings and the Sweedler product}

In \cite[$\S$ 3.4]{AJ}, Anel and Joyal introduce the notion of Sweedler product as follows : given a coalgebra $C$ and an associative algebra $A$, a measuring $\phi:C\longrightarrow Hom_K(A,E)$ is said to be be universal if given any  measuring $\phi':C\longrightarrow Hom_K(A,B)$, there exists a unique morphism $f:E\longrightarrow B$ of algebras such that 
$\phi'=Hom(A,f)\circ \phi$. In \cite[Theorem 3.4.1]{AJ}, it is shown that such a universal measuring always exists and the algebra $E$ is referred to as the Sweedler product $E=C\rhd A$. Further, the functor
\begin{equation}\label{sprodw}
C\rhd - : Alg_K\longrightarrow Alg_K 
\end{equation} is left adjoint to the functor $[C,-]:Alg_K\longrightarrow Alg_K$ that takes an algebra $A$ to its convolution algebra $Hom(C,A)$ with respect to $C$. 

\smallskip
In this section, we will   construct the Sweedler product $C\rhd\mathscr A$ of a coalgebra $C$ and an algebra $\mathscr A$ over the operad $\mathcal O$. Given a vector space $V$,
we can form the free $\mathcal O$-algebra $\mathscr F(V)$ over $V$ (see, for instance, \cite[$\S$ 5.2.5]{LoVa}) such that we have natural isomorphisms 
\begin{equation}\label{fadjoint}
Alg_{\mathcal O}(\mathscr F(V),\mathscr B) \cong Vect_K(V,\mathscr B)
\end{equation} for any $\mathcal O$-algebra $\mathscr B$. In particular, there is a canonical morphism $V\longrightarrow \mathscr F(V)$. We also recall the following definition.

\begin{defn}\label{D8.1} (see, for instance, \cite[Definition 1.8]{Liv}) Let $\mathscr B$ be an $\mathcal O$-algebra. An ideal $\mathscr I$ in $\mathscr B$ is a sub-vector space 
$\mathscr I\subseteq \mathscr B$ having the property that
\begin{equation}\label{ideall}
\theta(b_1,...,b_{n-1},x)\in \mathscr I\qquad \forall\textrm{ }\theta\in \mathcal O(n),  b_i\in \mathscr B, x\in \mathscr I
\end{equation}

\end{defn}

 We are now ready to construct   the Sweedler product $C\rhd \mathscr A$ for a coalgebra $C$ and an $\mathcal O$-algebra $\mathscr A$. We start with the  free $\mathcal O$-algebra $\mathscr F(C\otimes \mathscr A)$ over the vector space $C\otimes\mathscr A$. For an element $c\otimes a\in C\otimes \mathscr A$, we will denote by $c\rhd a$ its image under the canonical morphism $C\otimes \mathscr A\longrightarrow \mathscr F(C\otimes \mathscr A)$.
 
 \smallskip
  We now consider the smallest ideal $\mathscr J \subseteq \mathscr F(C\otimes \mathscr A)$ containing the elements
 \begin{equation}\label{cproda}
 \begin{array}{c}
c\rhd (u_{\mathscr A}(\theta_0))-\epsilon(c)u_{\mathscr F(C\otimes \mathscr A)}(\theta_0)\qquad c\rhd \theta(a_1,...,a_n) - \sum \theta (c_1\rhd a_1,...,c_n\rhd a_n)
 \end{array}
 \end{equation} for all $c\in C$, $a_i\in A$, $\theta_0\in \mathcal O(0)$, $\theta\in \mathcal O(n)$, $n\geq 0$. We set $C\rhd\mathscr A$ to be the quotient $\mathscr F(C\otimes
 \mathscr A)/\mathscr J$. By abuse of notation, the image of the element $c\rhd a\in \mathscr F(C\otimes \mathscr A)$ in the quotient $C\rhd \mathscr A$ will still be denoted
 by $c\rhd a$. 
 
 \begin{lem}\label{L8.2} The canonical morphism
 \begin{equation}
 \phi(C,\mathscr A)=\phi:C\longrightarrow Hom_K(\mathscr A,C\rhd \mathscr A) \qquad c\mapsto (a\mapsto c\rhd a)
 \end{equation}
  is a measuring of $\mathcal O$-algebras.
 \end{lem}

\begin{proof} We consider the obvious map 
 \begin{equation}
\phi':C\longrightarrow Hom_K(\mathscr A,\mathscr F(C\otimes  \mathscr A)) \qquad c\mapsto (a\mapsto c\rhd a)
 \end{equation} Applying the quotient map $\mathscr F(C\otimes
 \mathscr A)\longrightarrow \mathscr F(C\otimes
 \mathscr A)/\mathscr J=C\rhd \mathscr A$, we obtain 
  \begin{equation}
 \phi(C,\mathscr A)=\phi:C\longrightarrow Hom_K(\mathscr A,C\rhd \mathscr A) \qquad c\mapsto (a\mapsto c\rhd a)
 \end{equation}
We now consider $a_1$,...,$a_n\in \mathscr A$ and $\theta\in \mathcal O(n)$, $\theta_0\in \mathcal O(0)$. For any $c\in C$, we have
\begin{equation*}
\begin{array}{c}
\phi'(c)(\theta(a_1,...,a_n))-\sum \theta(\phi'(c_1)(a_1),...,\phi'(c_n)(a_n))=c\rhd \theta(a_1,...,a_n) -  \sum \theta (c_1\rhd a_1,...,c_n\rhd a_n)\in 
\mathscr J\\
\phi'(c)(u_{\mathscr A}(\theta_0))-\epsilon(c)u_{\mathscr F(C\otimes \mathscr A)}(\theta_0)=c\rhd (u_{\mathscr A}(\theta_0))-\epsilon(c)u_{\mathscr F(C\otimes \mathscr A)}(\theta_0)
\in \mathscr J\\
\end{array}
\end{equation*} Since $C\rhd\mathscr A=\mathscr F(C\otimes
 \mathscr A)/\mathscr J$, the result is clear.
\end{proof}

\begin{Thm}\label{Th8.3} Let $C$ be a coalgebra and let $\mathscr A$ be an $\mathcal O$-algebra. Then, the measuring $\phi(C,\mathscr A):C\longrightarrow Hom_K(\mathscr A,C\rhd \mathscr A)$ has the following universal property: given any
measuring $\phi':C\longrightarrow Hom_K(\mathscr A,\mathscr B)$ of $\mathcal O$-algebras, there exists a unique morphism $f:C\rhd \mathscr A\longrightarrow \mathscr B$ 
of $\mathcal O$-algebras such that
$\phi'=Hom(\mathscr A,f)\circ \phi(C,\mathscr A)$.
\end{Thm}

\begin{proof} It is clear that the morphism $\phi': C\longrightarrow Hom_K(\mathscr A,\mathscr B)$ corresponds to a map $\psi:C\otimes \mathscr A\longrightarrow \mathscr B$
of vector spaces. From the property of the free algebra in \eqref{fadjoint}, this induces a unique morphism $g:\mathscr F(C\otimes\mathscr A)\longrightarrow \mathscr B$ of $\mathcal O$-algebras
satisfying $\psi=g\circ i$, where $i:C\otimes \mathscr A\longrightarrow \mathscr F(C\otimes \mathscr A)$ is the canonical morphism.

\smallskip
Considering elements in  $\mathscr J \subseteq \mathscr F(C\otimes \mathscr A)$ as in \eqref{cproda}, we now have
\begin{equation*}
\begin{array}{ll}
g(c\rhd (u_{\mathscr A}(\theta_0)))-g(\epsilon(c)u_{\mathscr F(C\otimes \mathscr A)}(\theta_0))&=\psi(c\otimes (u_{\mathscr A}(\theta_0)))-\epsilon(c)u_{\mathscr B}(\theta_0)\\
&=\phi'(c)(u_{\mathscr A}(\theta_0))-\epsilon(c)u_{\mathscr B}(\theta_0)\\
&= 0\\
&\\
 g(c\rhd \theta(a_1,...,a_n)) - \sum g(\theta (c_1\rhd a_1,...,c_n\rhd a_n))&=\psi(c\otimes \theta(a_1,...,a_n)) - \sum \theta (g(c_1\rhd a_1),...,g(c_n\rhd a_n))\\
 &=\psi(c\otimes \theta(a_1,...,a_n)) - \sum \theta (\psi(c_1\otimes a_1),...,\psi(c_n\otimes a_n))\\
 &=\phi'(c)(\theta(a_1,...,a_n))-\sum \theta(\phi'(c_1)(a_1),...,\phi'(c_n)(a_n)) \\
 &=0\\
\end{array}
\end{equation*} As such, the morphism $g:\mathscr F(C\otimes\mathscr A)\longrightarrow \mathscr B$ factors uniquely as $g=f\circ q$, where  
$f:C\rhd\mathscr A\longrightarrow \mathscr B$ is a morphism of $\mathcal O$-algebras and $q$ is the quotient map $q:\mathscr F(C\otimes\mathscr A)\longrightarrow C\rhd
\mathscr A$. 

\smallskip
  For $c\in C$ and $a\in \mathscr A$, we see that
\begin{equation}\label{eq8.9y}
f(\phi(C,\mathscr A)(c)(a))=f(q(c\rhd a))=g(c\rhd a)=g(i(c\otimes a))=\psi(c\otimes a)=\phi'(c)(a)
\end{equation} This shows that $\phi'=Hom(\mathscr A,f)\circ \phi(C,\mathscr A)$. 

\smallskip 
Finally, suppose that $f_1,f_2:C\rhd \mathscr A\longrightarrow \mathscr B$ are morphisms 
of $\mathcal O$-algebras such that
$\phi'=Hom(\mathscr A,f_1)\circ \phi(C,\mathscr A)=Hom(\mathscr A,f_2)\circ \phi(C,\mathscr A)$. Then, the reasoning in \eqref{eq8.9y} shows that
\begin{equation}
f_1(q(c\rhd a))=\phi'(c)(a)=f_2(q(c\rhd a))
\end{equation} In other words, the restrictions $(f_1q)|(C\otimes \mathscr A)$ and $(f_2q)|(C\otimes \mathscr A)$ coincide. The universal property
of $\mathscr F(C\otimes \mathscr A)$ now gives $f_1q=f_2q$. Since $q$ is an epimorphism, it now follows that $f_1=f_2$. This proves the result.

\end{proof}

The following fact is possibly well-known: if $C$ is a cocommutative coalgebra and $\mathscr A$ is an $\mathcal O$-algebra, the space $[C,\mathscr A]:=Hom(C,\mathscr A)$ of linear maps carries the structure of an
$\mathcal O$-algebra given by setting
\begin{equation}\label{convolax}
\theta(f_1,...,f_n)(c):=\theta(f_1(c_1),...,f_n(c_n))\qquad \theta_0(c)=\epsilon(c)u_{\mathscr A}(\theta_0)
\end{equation} for $f_1$,...,$f_n\in [C,\mathscr A]$, $\theta\in \mathcal O(n)$, $n\geq 1$ , $c\in C$ and $\theta_0\in \mathcal O(0)$. We conclude with the following adjunction result.

\begin{Thm}\label{Th8.4} Let $C$ be a cocommutative coalgebra. For $\mathcal O$-algebras $\mathscr A$, $\mathscr B$, there are natural isomorphisms
\begin{equation}
Alg_{\mathcal O}(C\rhd \mathscr A,\mathscr B)=Alg_{\mathcal O}(\mathscr A,[C,\mathscr B])
\end{equation}
\end{Thm}

\begin{proof}
We start with a morphism $f:C\rhd \mathscr A\longrightarrow \mathscr B$ of $\mathcal O$-algebras. Composing with the canonical morphism $C\otimes \mathscr A
\longrightarrow \mathscr F(C\otimes \mathscr A)\longrightarrow C\rhd\mathscr A$, we obtain a map $f':C\otimes \mathscr A\longrightarrow \mathscr B$ of vector spaces.
Then, $f'$ corresponds to a morphism $f'':\mathscr A\longrightarrow Hom(C,\mathscr B)=[C,\mathscr B]$ of vector spaces. For $\theta\in \mathcal O(n)$, $a_1$,...,$a_n\in A$ and $c\in C$, we now see that
\begin{equation}\label{adjone}
\begin{array}{ll}
f''(\theta(a_1,...,a_n))(c) &=f'(c\otimes \theta(a_1,...,a_n))\\&=f(c\rhd \theta(a_1,...,a_n)) \\&\\
\theta(f''(a_1),...,f''(a_n))(c)&=\sum\theta(f''(a_1)(c_1),...,f''(a_n)(c_n))\\ &=\sum\theta(f'(c_1\otimes a_1),...,f'(c_n\otimes a_n))\\
&= \sum\theta(f(c_1\rhd a_1),...,f(c_n\rhd a_n))\\
&=f(\sum \theta (c_1\rhd a_1,...,c_n\rhd a_n)) \\
\end{array}
\end{equation} Since $c\rhd \theta(a_1,...,a_n)=\sum \theta (c_1\rhd a_1,...,c_n\rhd a_n)$ in $C\rhd\mathscr A$, it is clear from \eqref{adjone} that $f''$ is a morphism
of $\mathcal O$-algebras.

\smallskip
Conversely, suppose that we have a morphism $g:\mathscr A\longrightarrow [C,\mathscr B]$ of $\mathcal O$-algebras. This induces a morphism $g':C\otimes \mathscr A
\longrightarrow \mathscr B$ of vector spaces. The universal property of the free $\mathcal O$-algebra $\mathscr F(C\otimes \mathscr A)$ now leads to a morphism 
$g'':\mathscr F(C\otimes \mathscr A)\longrightarrow \mathscr B$ of $\mathcal O$-algebras such that $g''\circ i=g'$, where $i:C\otimes \mathscr A\longrightarrow \mathscr F(C\otimes \mathscr A)$
is the canonical morphism.

\smallskip
Considering elements in  $\mathscr J \subseteq \mathscr F(C\otimes \mathscr A)$ as in \eqref{cproda}, we now have
\begin{equation*}
\begin{array}{ll}
g''(c\rhd (u_{\mathscr A}(\theta_0)))-g''(\epsilon(c)u_{\mathscr F(C\otimes \mathscr A)}(\theta_0))&=(g''\circ i)(c\otimes (u_{\mathscr A}(\theta_0)))-\epsilon(c)u_{\mathscr B}(\theta_0)\\
&=g'(c\otimes (u_{\mathscr A}(\theta_0)))-\epsilon(c)u_{\mathscr B}(\theta_0)\\
&= g(u_{\mathscr A}(\theta_0))(c)-\epsilon(c)u_{\mathscr B}(\theta_0)\\
&= (u_{[C,\mathscr B]}(\theta_0))(c)-\epsilon(c)u_{\mathscr B}(\theta_0)\\
&= 0\\
&\\
 g''(c\rhd \theta(a_1,...,a_n)) - \sum g''(\theta (c_1\rhd a_1,...,c_n\rhd a_n))&=g'(c\otimes \theta(a_1,...,a_n)) - \sum \theta (g''(c_1\rhd a_1),...,g''(c_n\rhd a_n))\\
 &= g(\theta(a_1,...,a_n))(c) - \sum \theta (g'(c_1\otimes  a_1),...,g'(c_n\otimes a_n))\\
&= \theta(g(a_1),...,g(a_n))(c) - \sum \theta ( g(a_1)(c_1),..., g(a_n)(c_n))\\
 &=0\\
\end{array}
\end{equation*} It follows that $g'':\mathscr F(C\otimes \mathscr A)\longrightarrow \mathscr B$ descends to a morphism $g''':C\rhd \mathscr A\longrightarrow \mathscr B$
of $\mathcal O$-algebras. It may be easily verified that the two above associations are inverse to each other. This proves the result.
\end{proof}

\end{document}